\documentclass[11pt]{amsart}
\usepackage{amsmath,amssymb,amsthm}
\usepackage[bookmarks=true,colorlinks,linkcolor=blue,citecolor=blue]{hyperref}
\usepackage{graphicx,subcaption}
\usepackage{fullpage,comment}
\usepackage[shortlabels]{enumitem}
\setlist{font=\normalfont}
\usepackage{mathrsfs}
\usepackage{tikz,tikz-cd}
\usetikzlibrary{knots,hobby}

\usepackage{biblatex}
\providecommand{\abs}[1]{\left|{#1}\right|}
\newcommand{\sslash}{\mathbin{/\mkern-6mu/}}

\DeclareMathOperator{\tr}{tr}
\DeclareMathOperator{\id}{id}
\DeclareMathOperator{\order}{ord}

\DeclareMathOperator{\Span}{span}
\DeclareMathOperator{\Mod}{Mod}
\DeclareMathOperator{\MaxSpec}{MaxSpec}
\DeclareMathOperator{\Irr}{Irr}
\DeclareMathOperator{\cs}{cs}

\DeclareMathOperator{\End}{End}

\newcommand{\term}[1]{\textit{#1}}

\newcommand{\cx}{\mathbb{C}}
\newcommand{\ints}{\mathbb{Z}}
\newcommand{\reals}{\mathbb{R}}
\newcommand{\iunit}{\mathbf{i}}

\newcommand{\skein}{\mathscr{S}}
\newcommand{\evenS}{\skein^\mathrm{ev}}
\newcommand{\surface}{\Sigma}
\newcommand{\surcl}{\overline{\Sigma}}
\newcommand{\ptorus}{\Sigma_{1,1}}
\newcommand{\hsphere}{\Sigma_{0,4}}

\newcommand{\centV}{\mathcal{Z}}
\newcommand{\charV}{\mathcal{X}}
\newcommand{\azuV}{\mathcal{A}}
\newcommand{\exceptV}{\mathcal{E}}

\newcommand{\ideal}[1]{\langle#1\rangle}
\newcommand{\Fr}[1]{{\tilde{#1}}}
\newcommand{\Ch}{\Phi}

\newcommand{\tildei}{\tilde{\imath}}
\newcommand{\ram}{{\mathrm{ram}}}
\newcommand{\red}{{\mathrm{red}}}

\newtheorem{theorem}{Theorem}[section]
\newtheorem{thm}{Theorem}
\newtheorem{lemma}[theorem]{Lemma}
\newtheorem{proposition}[theorem]{Proposition}

\theoremstyle{remark}
\newtheorem{remark}{Remark}
\theoremstyle{definition}
\newtheorem{definition}{Definition}

\newcommand{\picmargin}{\mathop{}\!}
\newenvironment{linkdiag}{\picmargin\begin{tikzpicture}[scale=0.8,baseline=(ref.base)]}{\node (ref) at (0.5,0.5){\phantom{$-$}};\end{tikzpicture}\picmargin}

\begin{document}

\title{Explicit representations and Azumaya loci of skein algebras of small surfaces}
\author[Tao Yu]{Tao Yu}
\address{Shenzhen International Center for Mathematics, Southern University of Science and Technology, 1088 Xueyuan Avenue, Shenzhen, Guangdong, China}
\email{yut6@sustech.edu.cn}

\begin{abstract}
We construct finite dimensional representations of the Kauffman bracket skein algebra of the one-punctured torus and four-punctured sphere at all roots of unity. The representations are given by explicit formulas. They all have dimensions equal to the PI degrees of the skein algebras, and they realize all classical shadows. We then use the reducibility of these representations to determine the Azumaya loci. In particular, the Azumaya loci of these surfaces contain the smooth loci of the classical shadow varieties, with equality in the case of the one-punctured torus and proper containment in the case of the four-punctured sphere.
\end{abstract}

\maketitle

\section{Introduction}


Let $q\in\cx^\times$ be a nonzero complex parameter. The (Kauffman bracket) skein algebra $\skein_q(\surface)$ of a surface $\surface$, introduced by Przytycki \cite{Pr} and Turaev \cite{Tu}, is a $\cx$-algebra spanned by framed unoriented links in the thickened surface $\surface\times(-1,1)$ modulo the Kauffman bracket relations
\begin{align}
\begin{linkdiag}
\fill[gray!20!white] (-0.1,0)rectangle(1.1,1);
\begin{knot}[clip width=8,background color=gray!20!white]
\strand[very thick] (1,1)--(0,0);
\strand[very thick] (0,1)--(1,0);
\end{knot}
\end{linkdiag}
&=q
\begin{linkdiag}
\fill[gray!20!white] (-0.1,0)rectangle(1.1,1);
\draw[very thick] (0,0)..controls (0.5,0.5)..(0,1);
\draw[very thick] (1,0)..controls (0.5,0.5)..(1,1);
\end{linkdiag}
+q^{-1}
\begin{linkdiag}
\fill[gray!20!white] (-0.1,0)rectangle(1.1,1);
\draw[very thick] (0,0)..controls (0.5,0.5)..(1,0);
\draw[very thick] (0,1)..controls (0.5,0.5)..(1,1);
\end{linkdiag},\label{eq-skein}\\
\begin{linkdiag}
\fill[gray!20!white] (0,0)rectangle(1,1);
\draw[very thick] (0.5,0.5)circle(0.3);
\end{linkdiag}
&=(-q^2-q^{-2})
\begin{linkdiag}
\fill[gray!20!white] (0,0)rectangle(1,1);
\end{linkdiag}.\label{eq-loop}
\end{align}
The skein algebra is connected to many areas in low dimensional topology. Quantizations of character varieties \cite{Bul,BFK,PS} and Teichm\"uller spaces \cite{BWqtr,Mu} are a few examples.

We are interested in the representations of skein algebras. (All representations are finite dimensional in this paper.) When $q$ is a root of unity, the representation theory has a rich structure. First we recall some general theory. Let $S$ be a finitely generated $\cx$-algebra without zero divisors, and let $Z$ be the center of $S$. Suppose $X$ is a subalgebra of $Z$ such that $S$ is finitely generated as an $X$-module. Let $D^2=\dim_{\tilde{Z}}(S\otimes_Z\tilde{Z})<\infty$, where $\tilde{Z}$ is the field of fractions of $Z$. This dimension is always a perfect square, and its square root $D$ is called the \term{PI degree} of $S$.

Any irreducible representation of $S$ restricts to a representation of $Z$ and, by Schur's lemma, defines an algebra homomorphism $Z\to\cx$ called the \term{classical shadow}. This is a point in $\centV:=\MaxSpec Z$, which we call the \term{classical shadow variety}.

\begin{thm}[See e.g. \cite{BG,BY}]
Let $\Irr(S)$ be the set of isomorphism classes of irreducible representations of $S$ and $\cs:\Irr(S)\to\centV$ be the map sending an irreducible representation to its classical shadow. Then $\cs$ is surjective.

Every irreducible representation of $S$ has dimension at most $D$, the PI degree. Let $\Irr_D(S)\subset\Irr(S)$ be the set of $D$-dimensional irreducible representations. Then $\cs$ is injective on $\Irr_D(S)$, and $\azuV=\cs(\Irr_D(S))\subset\centV$ is open dense.
\end{thm}

The set $\azuV$ is called the \term{Azumaya locus} of $S$. Since $Z$ is a finite extension of $X$, there is a finite map $\pi:\centV\to\charV:=\MaxSpec X$. The subset $\azuV_0=\{x\in\charV\mid \pi^{-1}(x)\subset\azuV\}$ is called the \term{fully Azumaya locus}. Strictly speaking, the fully Azumaya locus depends on the choice of $X$.

Now we specialize to skein algebras. 
By \cite{Bulfg,PS2,FKL1}, when $q$ is a root of unity, $S=\skein_q(\surface)$ satisfy the conditions above. There is a natural choice of $X$ given in Section~\ref{sec-center}. In particular, $\charV$ is isomorphic to the $SL_2$-character variety
\begin{equation*}
\chi(\surface)=\{\text{homomorphisms }\pi_1(\surface)\to SL_2\}\sslash SL_2
\end{equation*}
or a quotient of it. The variety $\charV$ has a partition investigated in \cite{FKL3}. Let $\surface=\surface_{g,v}$ be the orientable surface of genus $g$ with $v$ punctures, and let $p_1,\dotsc,p_v$ be the peripheral curves around the punctures. For $\vec{W}\in\cx^v$, setting the trace around $p_i$ to be $W_i$ for all $i=1,\dotsc,v$ defines a sliced character variety, denoted $\charV_{\vec{W}}$.

\begin{thm}[\cite{FKL3}]
The smooth locus of $\charV_{\vec{W}}$ is included in the fully Azumaya locus of $\skein_q(\surface_{g,v})$ for all $\vec{W}\in\cx^v$.
\end{thm}

The main goal of this paper is to determine the Azumaya loci of two of the building blocks of surfaces: the one-punctured torus $\ptorus$ and the four-punctured sphere $\hsphere$.

\begin{thm}[Theorems~\ref{thm-main-ptorus} and \ref{thm-main-hsphere}]
Suppose $q$ is a root of unity.
\begin{enumerate}
\item For both $\surface=\ptorus$ and $\hsphere$, and for every $z\in\centV$, there exists a representation of $\skein_q(\surface)$ with classical shadow $z$ and dimension equal to the PI degree.
\item For both $\surface=\ptorus$ and $\hsphere$, the fully Azumaya locus of $\skein_q(\surface)$ is the union of the smooth loci of the slices $\charV_{\vec{W}}$ for all $\vec{W}\in\cx^v$.
\item The Azumaya locus of $\skein_q(\ptorus)$ is the smooth locus of $\centV$.
\item The Azumaya locus of $\skein_q(\hsphere)$ properly contains the smooth locus of $\centV$. The complement $\azuV^c$ of the Azumaya locus has one less component than the singular locus.
\end{enumerate}
\end{thm}

Our proof is constructive, and thus independent of the approach of \cite{FKL3}, which utilizes the theory of Poisson orders. In particular, the existence in (1) is proved with explicit formulas of representations, and the Azumaya loci are obtained by investigating the reducibility of these representations.

Some prior constructive results include \cite{HP} and \cite{Tak}. Many of our results are similar to theirs, which is inevitable due to the uniqueness over Azumaya loci. However, the conditions of our results are different, and this allowed us to fill many edge cases.

\section{Kauffman bracket skein algebra}

In this section, we review the definition and basic properties of the Kauffman bracket skein algebra of a surface $\surface$. We assume $\surface$ is an oriented surface obtained by removing zero or finitely many points (punctures) from a closed surface $\surcl$. We also exclude the sphere with at most 3 punctures.

\subsection{Definition}

Let $\surface$ be a surface and $M=\surface\times(-1,1)$. A (framed) link $\alpha$ in $M$ is a closed 1-dimensional submanifold of $M$ with a transverse vector field along it. Two links are isotopic if they are homotopic in the class of links. The empty set, by convention, is a link which is isotopic only to itself. As usual, links are depicted by diagrams on $\surface$, and the framing of a diagram is vertical.

For a nonzero $q\in\cx$, let $\skein_q(\surface)$ be the vector space over $\cx$ spanned by isotopy classes of links modulo the skein relation \eqref{eq-skein} and the trivial loop relation \eqref{eq-loop}.
For two links $\alpha$ and $\beta$, the product $\alpha\beta$ is defined as the result of stacking $\alpha$ above $\beta$. This is extended linearly to make $\skein_q(\surface)$ a $\cx$-algebra.

A diagram on $\surface$ has a well-defined homology class in $H:=H_1(\surface;\mathbb{Z}/2)$, and the defining relations \eqref{eq-skein} and \eqref{eq-loop} preserves this homology classes. Clearly, stacking diagrams has the effect of addition on homology. Thus, $\skein_q(\surface)$ is an $H$-graded algebra. Let $H^\ast=H^1(\surface;\ints/2)$. Then $\skein_q(\surface)$ has an $H^\ast$-action given by
\begin{equation}
\phi\cdot\alpha=(-1)^{\phi(h)}\alpha
\end{equation}
for $\phi\in H^\ast$ and a homogeneous element $\alpha\in\skein_q(\surface)$ with degree $h\in H$.

Let $\bar{H}=H_1(\surcl;\ints/2)$. The inclusion $\surface\hookrightarrow\surcl$ induces $H\twoheadrightarrow\bar{H}$, which determines an $\bar{H}$-grading of $\skein_q(\surface)$. The subalgebra with $0$ grading in $\bar{H}$ is the \term{even} subalgebra, denoted $\evenS_q(\surface)$. Equivalently, $\evenS_q(\surface)$ is the invariant subalgebra under the action of $\bar{H}^\ast=H^1(\surcl;\ints/2)\hookrightarrow H^\ast$.

More geometrically, a simple diagram is in the even subalgebra if and only if its geometric intersection number is even with all closed curves on $\surface$. Note that peripheral curves, which are curves on $\surface$ that bound one-punctured disks, are even. This is why only the homology of $\surcl$ is considered.

\subsection{Skein algebra at $q=\pm1$}\label{sec-skein1}

Given two diagrams $\alpha,\beta$, the product $\beta\alpha$ differs from $\alpha\beta$ by changing the crossing between $\alpha$ and $\beta$. The skein relation \eqref{eq-skein} implies that the product is non-commutative in general. However, the crossing difference is not present when $q=\pm1$, so the skein algebra is commutative in this case.

When $q=-1$, the skein algebra is isomorphic to the algebra of regular functions on the $SL_2$-character variety \cite{Bul,PS} (the field $\cx$ is omitted throughout)
\begin{equation}
\chi(\surface)=\{\text{homomorphisms }\pi_1(\surface)\to SL_2\}\sslash SL_2
\end{equation}
where the $SL_2$-action is conjugation. A simple closed curve $\alpha\in\skein_{-1}(\surface)$ corresponds to the trace function $\tau_\alpha$ given by
\begin{equation}\label{eq-charV-tr}
\tau_\alpha(r)=-\tr(r(\alpha)),\qquad (r:\pi_1(\surface)\to SL_2)\in\charV(\surface).
\end{equation}

When $q=1$, the skein algebra is also isomorphic to $\skein_{-1}(\surface)$. However, the isomorphism is not canonical: it depends on a choice of a spin structure on $\surface$. See \cite{Ba}.

There is a similar isomorphism result for the even subalgebra.

\begin{theorem}[\cite{FKL1}, Theorem~2.6]\label{thm-even-iso}
Suppose $\epsilon^4=1$. There is an isomorphism $\evenS_\epsilon(\surface)\to\evenS_{-1}(\surface)$. If $\alpha$ is an even link with $l$ components, then the isomorphism maps $\alpha$ to $(\epsilon^2)^l\alpha$.
\end{theorem}

\subsection{Roots of Unity}

When $q$ is a root of unity, let
\begin{equation}
n=\order(q^2),\quad N=\order(q^4),\quad\text{ and }\epsilon=q^{N^2}.
\end{equation}
These notations will be used throughout the paper.

A useful identity that is easy to check case by case is
\begin{equation}\label{eq-zeta2N}
(-q^2)^N=-\epsilon^2.
\end{equation}

\subsection{Chebyshev homomorphism}

An important tool to describe central elements of the skein algebra is the Chebyshev homomorphism. The \term{Chebyshev homomorphism} $\Ch:\skein_\epsilon(\surface)\to\skein_q(\surface)$ is an algebra embedding introduced in \cite{BW}. It is characterized by the effect on simple closed curves. Given a simple closed curve $\alpha$, let $\alpha$ denote the corresponding element in $\skein_q(\surface)$ and $\hat\alpha$ denote the corresponding element in $\skein_\epsilon(\surface)$. Then
\begin{equation}
\Ch(\hat\alpha)=T_N(\alpha),
\end{equation}
where $T_N(x)\in\ints[x]$ is the Chebyshev polynomial characterized by the equation
\begin{equation}
T_N(t+t^{-1})=t^N+t^{-N}.
\end{equation}
Alternatively, $T_N(x)$ is defined recursively by $T_0(x)=2$, $T_1(x)=x$, and $T_k(x)=xT_{k-1}(x)-T_{k-2}(x)$.

\subsection{Center of the skein algebra}
\label{sec-center}

Peripheral curves are in the center of the skein algebra since they do not intersect other diagrams after isotopy. Using the basis of \cite{Pr}, it is easy to see that they generate a polynomial subalgebra. When $q$ is not a root of unity, this subalgebra is the entire center of the skein algebra \cite{PS2}.

When $q$ is a root of unity, the center is much bigger.

\begin{theorem}\label{thm-skein-center}
Suppose $\surface=\surface_{g,v}$ is the surface with genus $g$ and $v$ punctures. Let
\begin{equation}\label{eq-defX}
X=\begin{cases}
\skein_\epsilon(\surface),&n\text{ is odd},\\
\evenS_\epsilon(\surface),&n\text{ is even}.
\end{cases}
\end{equation}
Let $Z$ be the center of $\skein_q(\surface)$, and $\charV=\MaxSpec X$, $\centV=\MaxSpec Z$.
\begin{enumerate}
\item (\cite{FKL1}, Theorem~4.1) $Z$ is generated by the peripheral curves and $\Ch(X)$.
\item (Corollary of \cite{FKL1}, Theorem~5.2 and Remark~5.4) Let $\hat{p}_1,\dotsc,\hat{p}_v\in X$ be the peripheral curves. Then $Z$ has a presentation
\begin{equation}
Z=\Ch(X)[p_1,\dotsc,p_v]/\ideal{\Ch(\hat{p}_i)=T_N(p_i)}.
\end{equation}
The map $\centV\to\charV$ induced by $\Ch$ is a branched covering of degree $N^p$.
\item The PI degree of $\skein_q(\surface)$ is
\begin{equation}\label{eq-PIdeg}
D=\begin{cases}
N^{3g-3+v},&n\text{ is odd},\\
2^gN^{3g-3+v},&n\text{ is even}.
\end{cases}
\end{equation}
\end{enumerate}
\end{theorem}

Using results in Section~\ref{sec-skein1}, $X$ is always isomorphic to the case $\epsilon=-1$. If $n$ is odd, then $\charV$ is isomorphic to the character variety $\chi(\surface)$. If $n$ is even, then $X\cong\evenS_{-1}(\surface)$ is the invariant subalgebra of $\skein_{-1}(\surface)$ under the action of $\bar{H}^\ast$, so $\charV$ is the quotient of $\chi(\surface)$ by $\bar{H}^\ast$. In either case, points in $\charV$ are equivalent classes of homomorphisms $\pi_1(\surface)\to SL_2$. We call the equivalent classes \term{characters}. By abuse of notation, if $z\in\centV$ maps to the character of $r:\pi_1(\surface)\to SL_2$ in $\charV$ by $\Ch$, we also say $z$ is a character of $r$.


\subsection{Sliced skein algebra}


Let $p_1,\dotsc,p_v$ be the peripheral curves of $\surface=\surface_{g,v}$. In \cite{FKL3}, the \term{sliced skein algebra} of $\surface$ at $\vec{w}=(w_1,\dotsc,w_v)\in\cx^v$ is defined as
\begin{equation}
\skein_{q,\vec{w}}(\surface)=\skein_q(\surface)/\ideal{p_i=w_i}.
\end{equation}

Now suppose $q$ is a root of unity. The Chebyshev homomorphism descends to a homomorphism
\begin{equation}
\Ch_{\vec{w}}:\skein_{\epsilon,\vec{W}}(\surface)\to\skein_{q,\vec{w}}(\surface),
\end{equation}
where $\vec{W}=(T_N(w_1),\dotsc,T_N(w_v))\in\cx^v$. Using the basis in \cite[Proposition~6.9]{FKL3}, it is easy to show that $\Ch_{\vec{w}}$ is an embedding.

\begin{theorem}[{\cite[Theorem~10.1]{FKL3}}]\label{thm-sliced-center}
Suppose $q$ is a root of unity. 
\begin{enumerate}
\item The center $Z_{\vec{w}}$ of the sliced skein algebra $\skein_{q,\vec{w}}(\surface)$ is the image of the center $Z$ of the skein algebra $\skein_q(\surface)$ under the natural quotient.
\item Let $X_{\vec{W}}\subset\skein_{\epsilon,\vec{W}}(\surface)$ be the natural quotient of $X$ in \eqref{eq-defX}. Then $Z_{\vec{w}}=\Ch_{\vec{w}}(X_{\vec{W}})$.
\item The PI degree of the sliced skein algebra is the same as the skein algebra, given in \eqref{eq-PIdeg}.
\end{enumerate}
\end{theorem}


Let $\centV_{\vec{w}}=\MaxSpec Z_{\vec{w}}$ and $\charV_{\vec{W}}=\MaxSpec X_{\vec{W}}$. Then $\centV_{\vec{w}}$ is the closed subset of $\centV$ cut out by the equations $p_i=w_i$, and similarly, $\charV_{\vec{W}}$ is the closed subset of $\charV$ cut out by $\hat{p}_i=W_i$. The map $\centV\to\charV$ induced by $\Ch$ restricts to an isomorphism $\centV_{\vec{w}}\to\charV_{\vec{W}}$, which is induced by $\Ch_{\vec{w}}$.


\section{Skein algebra of the one-punctured torus}

The one-punctured torus $\ptorus$ is the simplest open surface of which the skein algebra is not commutative. We review some relevant facts as well as specialize some results in the previous sections to $\ptorus$.

\subsection{Presentation}

Let $\ints^2\subset\reals^2$ be the standard lattice. The one-punctured torus can be modeled as
\begin{equation*}
\ptorus=(\reals^2\setminus\ints^2)/\ints^2.
\end{equation*}
Let $\alpha_m$ denote the simple closed curve that lifts to a straight line with slope $m$. The diffeomorphisms of $\ptorus$ acts transitively on the set of rational slopes.

The skein algebra of the one-punctured torus $\skein_q(\ptorus)$ is generated by $\alpha_0,\alpha_1,\alpha_\infty$ with the relations \cite{BP}
\begin{align}
q\alpha_0\alpha_\infty-q^{-1}\alpha_\infty\alpha_0&=(q^2-q^{-2})\alpha_1,
\label{eq-rel1}\\
q\alpha_1\alpha_0-q^{-1}\alpha_0\alpha_1&=(q^2-q^{-2})\alpha_\infty,
\label{eq-relinf}\\
q\alpha_\infty\alpha_1-q^{-1}\alpha_1\alpha_\infty&=(q^2-q^{-2})\alpha_0.
\label{eq-rel0}
\end{align}
If $q$ is not a 4th root of unity, then $\alpha_1$ is a redundant generator by \eqref{eq-rel1}. 

The peripheral curve $p$ is represented by
\begin{equation}\label{eq-boundary}
p=q\alpha_0\alpha_\infty\alpha_1-q^2\alpha_0^2-q^2\alpha_1^2-q^{-2}\alpha_\infty^2+q^2+q^{-2}.
\end{equation}


\subsection{Center at a root of unity}

Next consider the root of unity $q$. For $\ptorus$, the PI degree $D=n$ by \eqref{eq-PIdeg}. To avoid confusion, elements of $\skein_\epsilon(\ptorus)$ are denoted with a hat. Define $A_m=T_D(\alpha_m)=\Ch(T_{D/N}(\hat{\alpha}_m))\in Z$.

First consider the case when $n$ is odd. Then $\epsilon=\pm1$. $\hat{\alpha}_0,\hat{\alpha}_1,\hat{\alpha}_\infty$ generate $X$, and the relations \eqref{eq-rel1}--\eqref{eq-rel0} simply express commutativity. Thus, $X=\cx[\hat{\alpha}_0,\hat{\alpha}_1,\hat{\alpha}_\infty]$, but for consistency, we consider $\hat{p}$ as a generator of $X$ as well. $Z$ is generated by $A_0,A_1,A_\infty,p$ with the relation
\begin{equation}\label{eq-rel-odd}
A_0^2+A_1^2+A_\infty^2-\epsilon A_0A_\infty A_1+T_N(p)-2=0.
\end{equation}

Now look at the case when $n$ is even (so $D=2N$). Then $X=\evenS_\epsilon(\ptorus)$ is the invariant subalgebra of $\skein_\epsilon(\ptorus)$ under the action of $\bar{H}^\ast$. We can show that $X$ is generated by the invariant combinations $\hat\alpha_0^2,\hat\alpha_1^2,\hat\alpha_\infty^2,\hat{p}$ with the relation 
\begin{equation*}
\hat\alpha_0^2\hat\alpha_1^2\hat\alpha_\infty^2
=\epsilon^2(\hat\alpha_0\hat\alpha_\infty\hat\alpha_1)^2
=(\hat\alpha_0^2+\hat\alpha_1^2+\hat\alpha_\infty^2+\epsilon^2\hat{p}-2)^2.
\end{equation*}
By definition, $A_m=\Ch(\hat\alpha_m^2-2)$. Then $Z$ is generated by $A_0,A_1,A_\infty,p$ with the relation
\begin{equation}\label{eq-rel-even}
(A_0+2)(A_1+2)(A_\infty+2)=(A_0+A_1+A_\infty+\epsilon^2T_N(p)+4)^2.
\end{equation}

The presentations define $\centV=\MaxSpec Z$ as a closed subset of $\cx^4$, and $A_0,A_1,A_\infty,p$ correspond to the coordinate functions of $\cx^4$ restricted to $\centV$. We maintain a clear distinction between the coordinates of a point and the coordinate functions. For example, the point $z=(z_0,z_1,z_\infty,w)\in\centV\subset\cx^4$ correspond to the maximal ideal $\ideal{A_0-z_0,A_1-z_1,A_\infty-z_\infty,p-w}\subset Z$, and $A_0(z)=z_0$, for example. We make exceptions for the peripheral curves since we sometimes consider the sliced skein algebra. This convention will be used in the rest of the paper.

\subsection{Singularities}\label{sec-sing}

The singularities of $\centV$ can be found as critical points of the defining equations. The calculations are standard, so we omit them here.

\begin{lemma}
Suppose $q$ is a root of unity with $n$ odd and $\surface=\ptorus$. Let $z=(z_0,z_1,z_\infty,w)\in\centV$. Then $z$ is singular in the slice $\centV_w$ if and only if one of the following holds.
\begin{enumerate}
\item $z_0=z_1=z_\infty=0$. (In this case, $T_N(w)=2$.)
\item $z_0,z_1,z_\infty=\pm2$ such that $z_0z_1z_\infty=8\epsilon$, (In this case, $T_N(w)=-2$).
\end{enumerate}
$z$ is singular in the whole $\centV$ if and only if $z$ is singular in the slice and $T'_N(w)=0$.
\end{lemma}

\begin{remark}
The singular points are characters of special representations $r:\pi_1(\ptorus)\to SL_2$. In the first case, the image of $r$ is the 8-element quaternion group. In the second case, $r$ is central, meaning the image is in $\{\pm I\}\subset SL_2$.
\end{remark}

\begin{lemma}
Suppose $q$ is a root of unity with $n$ even and $\surface=\ptorus$. Let $z=(z_0,z_1,z_\infty,w)\in\centV$. Then $z$ is singular in the slice $\centV_w$ if and only if one of the following holds.
\begin{enumerate}
\item At least two of $z_0,z_1,z_\infty$ are $-2$. (In this case, the last one equals $-\epsilon^2T_N(w)$.)
\item $z_0=z_1=z_\infty=2$ and $T_N(w)=-2\epsilon^2$.
\end{enumerate}
$z$ is singular in the whole $\centV$ if and only if $z$ is in (1) or $z$ is in (2) and $T'_N(w)=0$.
\end{lemma}

\begin{remark}
There are more singular points in this case because of the quotient by $\bar{H}^\ast$. The first type of singular points are orbits with nontrivial stabilizers. The second type here is the orbit of the central representations.
\end{remark}

\subsection{Exceptional points}\label{sec-except}

A point $z\in\centV$ or $\charV$ is \term{exceptional} if $A_m(z)=\pm2$ for all slopes $m$. Clearly, if the projection of $z\in\centV$ in $\charV$ is exceptional, then $z$ is exceptional as well. As we will see in the next section, constructions are easier for non-exceptional points.

\begin{lemma}\label{lemma-except}
Suppose $z=(z_0,z_1,z_\infty,W)\in\charV$ such that $z_0,z_1,z_\infty$ are all $\pm2$. Then $z$ is exceptional if and only if one of the following holds.
\begin{enumerate}
\item $n$ is odd, and $z_0z_1z_\infty=8\epsilon$.
\item $n$ is even, and at least two of $z_0,z_1,z_\infty$ are $-2$.
\item $n$ is even, $z_0=z_1=z_\infty=2$, and $W=-2\epsilon^2$.
\end{enumerate}
\end{lemma}

\begin{proof}
For $\Leftarrow$, we identify the representations $r:\pi_1(\ptorus)\to SL_2$ whose characters are the points in the list. (1) and (3) correspond to central representations, so all traces are $\pm2$. In (2), if all three are $-2$, then the corresponding representation is quaternion. Otherwise, the image of the corresponding representation is cyclic of order $4$. The traces are $\pm2$ or $0$, so $A_m=\pm2$. Now we focus on $\Rightarrow$.

First, suppose $n$ is odd. By assumption, $z_0^2=z_1^2=z_\infty^2=4$ and $z_0z_1z_\infty=\pm8$. By \eqref{eq-boundary},
\begin{equation*}
W=\epsilon z_0z_1z_\infty-z_0^2-z_1^2-z_\infty^2+2=\pm8-10=0,
\end{equation*}
so either $z_0z_1z_\infty=8\epsilon$ and $W=-2$, or $z_0z_1z_\infty=-8\epsilon$ and $W=-18$. In the second case, the relation 
\begin{equation*}
\hat\alpha_0\hat\alpha_\infty=\epsilon\hat\alpha_1+\epsilon^{-1}\hat\alpha_{-1}=\epsilon(\hat\alpha_1+\hat\alpha_{-1})
\end{equation*}
in $\skein_\epsilon(\ptorus)$ implies $z_0z_\infty=\epsilon(z_1+A_{-1}(z))$, so
\begin{equation*}
-8=\epsilon(z_0z_\infty)z_1=z_1^2+z_1A_{-1}(z)=4\pm2A_{-1}(z).
\end{equation*}
Thus, $A_{-1}(z)=\pm6$, which contradicts the assumption $A_m(z)=\pm2$. Thus, only the first case is possible, which means $z$ is (1).

Next, suppose $n$ is even, and one of $z_0,z_1,z_\infty$ is $-2$. If exactly one of them is $-2$, then from the relation \eqref{eq-rel-even}
\begin{equation*}
(z_0+z_1+z_\infty+\epsilon^2W+4)^2=(z_0+2)(z_1+2)(z_\infty+2)=0,
\end{equation*}
we get $W=-6\epsilon^2$. Without loss of generality, we can assume $z_0=z_\infty=2$ and $z_1=-2$. A routine calculation in $\skein_\epsilon(\ptorus)$ yields
\begin{equation*}
(\hat\alpha_0^2-2)(\hat\alpha_\infty^2-2)=(\hat\alpha_1^2-2)+(\hat\alpha_{-1}^2-2)+2\epsilon^2\hat{p}+4,
\end{equation*}
which implies
\begin{equation*}
z_1+A_{-1}(z)+2\epsilon^2W+4=z_0z_\infty=4,\qquad
A_{-1}(z)=-z_1-2\epsilon^2W=14\ne\pm2.
\end{equation*}
This contradiction shows that at least two of $z_0,z_1,z_\infty$ are $-2$, which means $z$ is (2).

Finally, suppose $n$ is even and $z_0=z_1=z_\infty=2$. Then by \eqref{eq-rel-even},
\begin{equation*}
64=(z_0+2)(z_1+2)(z_\infty+2)=(z_0+z_1+z_\infty+\epsilon^2W+4)^2=(\epsilon^2W+10)^2,
\end{equation*}
so $W$ is $-2\epsilon^2$ or $-18\epsilon^2$. In the second case, the same calculation as above shows that $A_{-1}(z)=34\ne\pm2$, so only the first case is possible. This means $z$ is (3).
\end{proof}

\section{Representations for the one-punctured torus}

In this section, we assume that $q$ is not a 4th root of unity. This implies $D\ge 3$.

In \cite{HP}, all irreducible representations of $\skein_q(\ptorus)$ with $n$ odd and most irreducible representations with $n$ is even are obtained. However, the classical shadows are difficult to obtain. \cite{Tak} also gave irreducible representations covering a big part of the Azumaya locus when $n$ is odd.

We take a slightly different approach. We define classical shadows for representations that are possibly reducible. Suppose $\rho:S\to\End(V)$ is a representation of the algebra $S$ with center $Z$. A map $\rho_Z:Z\to\cx$ such that
\begin{equation}\label{eq-res-z}
\rho(z)=\rho_Z(z)\id_V\qquad\text{for all }z\in Z,
\end{equation}
is called the \term{classical shadow} of $\rho$. Classical shadows are clearly unique, but they do not always exist for reducible representations.

Instead of classifying irreducible representations, we look for $D$-dimensional representations that have classical shadows, and we allow the construction to give isomorphic representation. The uniformity of dimension and the flexibility make classical shadows much easier to calculate. 

\begin{theorem}\label{thm-main-ptorus}
Suppose $q$ is a root of unity. For every $z\in\centV$, there exists a $D$-dimensional representation $\rho:\skein_q(\ptorus)\to\End(V)$ with classical shadow $z$, and $\rho$ is reducible if and only if $z$ is singular in $\centV$.

The Azumaya locus of $\skein_q(\ptorus)$ is the smooth locus of $\centV$. The fully Azumaya locus of $\skein_q(\ptorus)$ is the union of the smooth loci of the slices of $\charV$.
\end{theorem}

\begin{proof}
The construction of representations is given in Propositions~\ref{prop-rep} and \ref{prop-gen-exist} when $A_0(z)\ne\pm2$. Using diffeomorphisms of the surface, we cover every non-exceptional point. The construction for exceptional points is given in Propositions~\ref{prop-except-odd}, and \ref{prop-except-even}.

The reducibility is given in Lemmas~\ref{lemma-type0-red}, \ref{lemma-red-2o}, and \ref{lemma-red-2e}.

The Azumaya locus follow directly from the first half of the theorem, and the fully Azumaya locus follows from the description of singularities in Section~\ref{sec-sing}.
\end{proof}

\subsection{Type-$0$ representations}

We start with a useful lemma.

\begin{lemma}\label{lemma-Cheb-prod}
Let $\omega$ be a primitive $D$-th root of unity. Then
\begin{equation}\label{eq-Cheb-prod}
\prod_{i=1}^D(\omega^it-\omega^{-i}t^{-1})
=\begin{cases}
t^D-t^{-D},&D\text{ is odd},\\
2-t^D-t^{-D},&D\text{ is even}.
\end{cases}
\end{equation}
\end{lemma}

\begin{proof}
We prove the identity in the field of rational functions $\cx(t)$. To avoid confusion, we use the boldface $\iunit$ to denote the imaginary unit. Let
\begin{equation*}
x_i=\omega^it-\omega^{-i}t^{-1},\qquad
f(x)=\prod_{i=1}^D(x-\iunit x_i).
\end{equation*}
By definition, $f(x)$ is a monic polynomial over $\cx(t)$ with roots $\iunit x_i$, each with multiplicity $1$. On the other hand,
\begin{equation*}
T_D(\iunit x_i)
=(\iunit\omega^it)^D+(\iunit^{-1}\omega^{-i}t^{-1})^D
=\iunit^Dt^D+\iunit^{-D}t^{-D}.
\end{equation*}
This implies that $f(x)=T_D(x)-\iunit^Dt^D-\iunit^{-D}t^D$. Thus,
\begin{align*}
\prod_{i=1}^D x_i&=\iunit^Df(0)
=\iunit^D(T_D(\iunit+\iunit^{-1})-\iunit^Dt^D-\iunit^{-D}t^{-D})\\
&=\iunit^D(\iunit^D+\iunit^{-D}-\iunit^Dt^D-\iunit^{-D}t^{-D})\\
&=(-1)^D+1-(-1)^Dt^D-t^{-D}.
\end{align*}
This agrees with \eqref{eq-Cheb-prod}.
\end{proof}

Write $z_0=\sigma^D+\sigma^{-D}$ and define
\begin{equation}
\lambda_i=q^{2i}\sigma+q^{-2i}\sigma^{-1},\qquad
\hat{\lambda}_i=q^{2i}\sigma-q^{-2i}\sigma^{-1}.
\end{equation}

By Lemma~\ref{lemma-Cheb-prod},
\begin{equation}\label{eq-lamhat-prodsq}
\prod_{i=1}^D\hat{\lambda}_i^2
=\begin{cases}
(\sigma^D-\sigma^{-D})^2=z_0^2-4,&n\text{ is odd},\\
(2-\sigma^D-\sigma^{-D})^2=(z_0-2)^2,&n\text{ is even}.
\end{cases}
\end{equation}
We say $z_0$ (or $\sigma$ by abuse of notation) is \term{general} if this is nonzero. More directly, $z_0$ is general if $z_0\ne\pm2$ or $z_0=-2$ and $n$ is even. In this case, define
\begin{align}
r_i&=r_i(\sigma,w):=\frac{w+q^{4i+2}\sigma^2+q^{-4i-2}\sigma^{-2}}{\hat{\lambda}_i\hat{\lambda}_{i+1}},\\
E^0_{\sigma,w}&=\{(s_1,\dotsc,s_D,t_1,\dotsc,t_D)\in\cx^{2D}\mid s_it_i=r_i(\sigma,w)\},
\end{align}
We will use the shorthands $\vec{s}=(s_1,\dotsc,s_D)$ and $\vec{t}=(t_1,\dotsc,t_D)$.

\begin{proposition}[Generalization of {\cite[Lemma~14]{Tak}}]\label{prop-rep}
Let $V$ be a $D$-dimensional vector space with a basis $\{v_i\mid i\in\ints/D\}$. Suppose $\sigma,w\in\cx$, $\sigma$ is general, and $(\vec{s},\vec{t})\in E^0_{\sigma,w}$. Then the following equations define a representation $\rho:\skein_q(\ptorus)\to\End(V)$.
\begin{align}
\rho(\alpha_0)v_i&=\lambda_iv_i,\label{eq-rep-0}\\
\rho(\alpha_1)v_i&=q^{2i+1}\sigma s_iv_{i+1}+q^{-2i+1}\sigma^{-1} t_{i-1}v_{i-1},\label{eq-rep-1}\\
\rho(\alpha_\infty)v_i&=s_iv_{i+1}+t_{i-1}v_{i-1}.\label{eq-rep-inf}
\end{align}
In addition, $\rho(A_0)=z_0$ and $\rho(p)=w$.
\end{proposition}

This representation will be called \term{type-$0$} since $\rho(\alpha_0)$ is diagonal. Using diffeomorphisms of the surface, we can replace the slope $0$ with any slope we want.

Note that replacing $\sigma$ by $q^{2i}\sigma$ amounts to cyclically relabeling the basis so that $v_i$ becomes $v_0$, and replacing $\sigma$ by $\sigma^{-1}$ amounts to reversing the order of labels.

\begin{proof}
$\rho(A_0)=z_0$ essentially by construction. The other identities to check are the presentation \eqref{eq-rel1}--\eqref{eq-rel0} and that the peripheral curve $p$ given by \eqref{eq-boundary} acts as the scalar $w$.

Each check is a routine calculation by applying basis vectors and comparing coefficients. The details are omitted.
\end{proof}

Define
\begin{gather}
R=R(z_0,w):=\begin{cases}
\dfrac{T_N(w)+z_0^2-2}{z_0^2-4},&n\text{ is odd},\\
\left(\dfrac{T_N(w)+\epsilon^2z_0}{z_0-2}\right)^2,&n\text{ is even}.
\end{cases}\label{eq-bigR}\\
\Pi_s(\vec{s},\vec{t})=\prod_{i=1}^Ds_i,\qquad
\Pi_t(\vec{s},\vec{t})=\prod_{i=1}^Dt_i.
\end{gather}

\begin{lemma}[Compare {\cite[Lemma~11]{Tak}}]\label{lemma-STimg}
Suppose $\sigma$ is general. Then the image of the map $(\Pi_s,\Pi_t):E^0_{\sigma,w}\to\cx^2$ is the curve $\{(x,y)\in\cx^2\mid xy=R(z_0,w)\}$.
\end{lemma}

\begin{proof}
To show that the image is contained in the curve $xy=R$, we just need to show
\begin{equation*}
R=\prod_{i=1}^Dr_i
=\prod_{i=1}^D\frac{w+q^{4i+2}\sigma^2+q^{-4i-2}\sigma^{-2}}{\hat{\lambda}_i\hat{\lambda}_{i+1}}.
\end{equation*}
The product of the denominators of $r_i$ agrees with the denominator of $R$ by \eqref{eq-lamhat-prodsq}. Now look at the numerator. For a fixed $\sigma$, the numerator of the product of $r_i$ is a degree $D$ monic polynomial in $w$ with roots $-q^{4i+2}\sigma^2-q^{-4i-2}\sigma^{-2}$. Note when $D$ is even, each root repeats twice. Since
\begin{gather*}
T_N(-q^{4i+2}\sigma^2-q^{-4i-2}\sigma^{-2})=(-q^{4i+2}\sigma^2)^N+(-q^{-4i-2}\sigma^{-2})^N\\
=(-q^2)^N(\sigma^{2N}+\sigma^{-2N})
=\begin{cases}
-(z_0^2-2),&D\text{ is odd},\\
-\epsilon^2z_0,&D\text{ is even},
\end{cases}
\end{gather*}
the numerator of $R$ is also monic and has the same roots, so the numerators agree.

Now we need to show the reverse inclusion. First suppose $x\ne0$. Set $s_1=x$, $s_i=1$ for $i\ne 1$, and $t_i=r_i/s_i$ for all $i$. Then by definition, $s_it_i=r_i$, and $\Pi_s(\vec{s},\vec{t})=x$. Then $\Pi_t(\vec{s},\vec{t})=R/x=y$. Thus, $(x,y)$ is the image of a point in $E^0_{\sigma,w}$.

Now suppose $x=0$. Then $R=0$, so $r_i=0$ for at least one $i$. For each $i$ such that $r_i=0$, set $s_i=0$, and leave $t_i$ undetermined for now. For all other $i$, set $s_i=1$ and $t_i=r_i$. Then $s_it_i=r_i$ for all $i$, and $\Pi_s(\vec{s},\vec{t})=0=x$. By construction, all $t_i$ determined so far are nonzero. Thus, we can choose the remaining $t_i$ to satisfy $\Pi_t(\vec{s},\vec{t})=y$. This shows $(x,y)$ is the image of a point in $E^0_{\sigma,w}$.
\end{proof}

To identify the classical shadow of a type-$0$ representation, we need to calculate powers of matrices. We can visualize the calculation using graphs. This perspective is also useful later for reducibility.

\begin{definition}\label{def-graph}
Let $A$ be a linear operator on the vector space with basis $\{v_i\}_{i\in I}$. Write $Av_j=\sum_{i\in I}a_{ij}v_i$. The \term{associated graph} (with respect to the basis $\{v_i\}_{i\in I}$) is a weighted and directed graph with vertex set $I$. For every pair $i,j\in I$ with $a_{ij}\ne0$, there is an edge $j\to i$ with weight $a_{ij}$. In other words, $j\to i$ is weighted by the $v_i$-component of $Av_j$. For convenience, edges with zero weight may be freely added or removed.
\end{definition}

For example, in a type-$0$ representation, the graph of $\rho(\alpha_\infty)$ has the following form
\begin{equation}\label{eq-graph-inf}
\begin{tikzcd}
\cdots \arrow[r,shift left] &
i-1 \arrow[r,shift left,"s_{i-1}"] \arrow[l,shift left]&
i \arrow[r,shift left,"s_i"] \arrow[l,shift left,"t_{i-1}"] &
i+1 \arrow[r,shift left,"s_{i+1}"] \arrow[l,shift left,"t_i"] &
i+2 \arrow[r,shift left] \arrow[l,shift left,"t_{i+1}"] &
\cdots \arrow[l,shift left]
\end{tikzcd}
\end{equation}

Using graphs, the powers of a matrix $A$ can be written as
\begin{equation}\label{eq-matpow}
A^kv_j=\sum_{i\in I}\Bigg(\sum_{p\in P^{(k)}_{ij}}w(p)\Bigg)v_i,
\end{equation}
where $P^{(k)}_{ij}$ is the set of paths from $j$ to $i$ of length $k$, and $w(p)$ is the product of the weights on the path. 

\begin{lemma}[Compare {\cite[Lemma~13]{Tak}}]\label{lemma-x1inf}
In a type-$0$ representation $\rho$, $A_1$ and $A_\infty$ acts as scalars. Therefore, the representation has a classical shadow.
\begin{enumerate}[(a)]
\item If $n$ is odd, then
\begin{equation}\label{eq-x1inf-a}
\rho(A_1)=\epsilon\sigma^D\Pi_s+\epsilon\sigma^{-D}\Pi_t,\qquad
\rho(A_\infty)=\Pi_s+\Pi_t,
\end{equation}
\item If $n$ is even, then
\begin{equation}\label{eq-x1inf-b}
\rho(A_1)=\sigma^D \Pi_s+\sigma^{-D} \Pi_t+f,\qquad
\rho(A_\infty)=\Pi_s+\Pi_t+f,
\end{equation}
where
\begin{equation}
f=f(z_0,w):=\frac{2\epsilon^2T_N(w)+4}{z_0-2}.
\end{equation}
\end{enumerate}
\end{lemma}

\begin{proof}
First suppose $z_0\ne\pm2$. This implies that the eigenvalues $\lambda_i$ of $\rho(\alpha_0)$ are distinct. Since $A_1,A_\infty$ are central, their actions commute with $\rho(\alpha_0)$. Thus, they must also be diagonal. Let
\begin{equation*}
\rho(A_1)v_i=z_1(i)v_i,\qquad \rho(A_\infty)v_i=z_\infty(i)v_i.
\end{equation*}
We show that $z_1(i),z_\infty(i)$ are given by the formulas above, which are independent of $i$. This implies that $A_1,A_\infty$ acts as scalars.

Consider $z_\infty=z_\infty(i)$. The associated graph of $\rho(\alpha_\infty)$ is a cycle going in both directions. See \eqref{eq-graph-inf}. Recall the exponents of all monomials in the Chebyshev polynomial $T_D(x)$ have the same parity as $D$. Then by \eqref{eq-matpow}, $z_\infty$ is a sum of weights of cycles based at $i$ with lengths at most $D$ and having the same parity as $D$. There are exactly two length $D$ cycles in the graph. They go all the way around the graph and contribute $\Pi_s+\Pi_t$ to $z_\infty$. For $z_1$, the graph is also a cycle but with different weights. The contribution is
\begin{equation}
\prod_{i=1}^Dq^{2i+1}\sigma s_i+\prod_{i=1}^Dq^{-2i-1}\sigma^{-1}t_i
=q^{D^2}\sigma^D\Pi_s+q^{-D^2}\sigma^{-D}\Pi_t.
\end{equation}

(a) Since $D=n$ is odd, any additional contribution must be from odd cycles. However, there are no other odd cycles, so there are no other terms. This proves \eqref{eq-x1inf-a}.

(b) Since $q^{D^2}=q^{4N^2}=1$, we can write
\begin{equation}\label{eq-x1inf-pre}
z_1=\sigma^D \Pi_s+\sigma^{-D} \Pi_t+f_1,\quad z_\infty=\Pi_s+\Pi_t+f_\infty
\end{equation}
so that $f_1$ and $f_\infty$, which depend on $i$ a priori, represent the contributions from other cycles. The cycles are too short to go all the way around, so they must back track. The edges of such a cycle can be rearranged into pairs $j\to j+1$ and $j+1\to j$, and each pair contributes weight $s_jt_j=r_j$. Therefore, $f_1$ and $f_\infty$ are polynomials of $r_j$, which are constant on $E^0_{\sigma,w}$.

Now apply \eqref{eq-rel-even} to the vector $v_i$ in the representation. Using \eqref{eq-x1inf-pre} and then $\Pi_s\Pi_t=R$, we get
\begin{align*}
0&=\left((\sigma^{-D}-\sigma^D)f_1-(1-\sigma^{2D})f_\infty-(1+\sigma^D)(2\epsilon^2T_N(w)+4)\right)\Pi_s+\\
&\quad+\left((\sigma^D-\sigma^{-D})f_1-(1-\sigma^{-2D})f_\infty-(1+\sigma^{-D})(2\epsilon^2T_N(w)+4)\right)\Pi_t+\\
&\quad+(\text{terms constant on }E^0_{\sigma,w}).
\end{align*}
This is a linear relation between $\Pi_s,\Pi_t$ and the coefficients are constant on $E^0_{\sigma,w}$. However, the image of $(\Pi_s,\Pi_t)$ cannot be contained in a line by Lemma~\ref{lemma-STimg}. Therefore, the coefficients of $\Pi_s$ and $\Pi_t$ are both zero. Solving for $f_1$ and $f_\infty$ give us $f_1=f_\infty=f$.

Finally, we deal with the case where $z_0=-2$ and $D=n$ is even. Define
\begin{equation*}
E^0=\bigcup_{\substack{\sigma\text{ general}\\w\in\cx}}E^0_{\sigma,w}\times\{(\sigma,w)\}\subset\cx^{2D+2}.
\end{equation*}
Matrix elements of $A_1$ and $A_\infty$ are continuous functions on $E^0$. The lemma is already proved for all slices $E^0_{\sigma,w}$ with $z_0\ne\pm2$. Then the remaining case follows from a continuity argument.
\end{proof} 

\begin{proposition}\label{prop-gen-exist}
There exists a type-$0$ representation for every point $(z_0,z_1,z_\infty,w)\in\centV$ with $z_0$ general.
\end{proposition}

\begin{proof}
If $z_0\ne\pm2$, then $\sigma^D\ne\pm1$. Let
\begin{equation}
\begin{pmatrix}x\\y\end{pmatrix}=\begin{pmatrix}\sigma^D&\sigma^{-D}\\1&1\end{pmatrix}^{-1}\begin{pmatrix}z'_1\\z'_\infty\end{pmatrix},\quad\text{where }
(z'_1,z'_\infty)=\begin{cases}
(\epsilon z_1,z_\infty),&n\text{ is odd},\\
(z_1-f,z_\infty-f),&n\text{ is even}.\end{cases}
\end{equation}
After some careful calculations, $xy=R$. If $z_0=-2$ and $n$ is even, we can show that $z_1+z_\infty=2f$ using \eqref{eq-rel-even}. In this case, let $(x,y)$ be a solution to the equations
\begin{equation}\label{eq-m2-exist}
x+y=z_\infty-f,\qquad xy=R.
\end{equation}
By Lemma~\ref{lemma-STimg}, there exists $e\in E^0_{\sigma,w}$ such that $\Pi_s(e)=x$, $\Pi_t(e)=y$. Then by Lemma~\ref{lemma-x1inf}, the classical shadow of the type-$0$ representation defined by $e$ is $(z_0,z_1,z_\infty,w)$.
\end{proof}

\subsection{Reducibility of type-$0$ representations}

\begin{lemma}\label{lemma-red}
Suppose $\rho:\skein_q(\ptorus)\to\End(V)$ is a type-$0$ representation.
\begin{enumerate}
\item If $\Pi_s=\Pi_t=0$ and $r_i=r_j=0$ for some $i\ne j$, then $\rho$ is reducible.
\item The converse is also true if $z_0\ne\pm2$.
\end{enumerate}
\end{lemma}

\begin{proof}
For $B\subset\ints/D$, let $V_B=\Span\{v_b\mid b\in B\}$.

(1) It is easy to see that the conditions are equivalent to $t_i=s_j=0$ (after possibly swapping $i$ and $j$). Then $V_{(i,j]}=\Span\{v_{i+1},\dotsc,v_j\}$ is an invariant subspace. It is proper and nontrivial since $i\ne j$. Thus, $\rho$ is reducible.

(2) Since $z_0\ne\pm2$, the eigenvalues of $\rho(\alpha_0)$ are distinct. Thus, any invariant subspace is of the form $V_B$. If $\rho$ is reducible, then there exists a proper and nontrivial $B$ such that $V_B$ is invariant. Thus, there exists $j\in B$ such that $j+1\notin B$. Since
\begin{equation}
\rho(\alpha_0-\lambda_{j-1})\rho(\alpha_\infty)v_j=\rho(\alpha_0-\lambda_{j-1})(s_jv_{j+1}+t_jv_{j-1})=(\lambda_{j+1}-\lambda_{j-1})s_jv_{j+1},
\end{equation}
$s_j=0$. Similarly, there exists $i\notin B$ (so $i\ne j$) such that $i+1\in B$, which implies $t_i=0$.
\end{proof}

\begin{remark}\label{rem-red}
The proof has a simple visualization using the graph associated to $\alpha_\infty$. Following the argument of the proof, if $V_B$ is an invariant subspace for some $B\subset\ints/D$, then there are no arrows coming out of $B$. Therefore, the graph fails to be strongly connected.

The graph of $\rho(\alpha_\infty)$ is shown in \eqref{eq-graph-inf}. If all edges have nonzero weights, then the graph is strongly connected. The condition $t_i=s_j=0$ removes enough edges to break strong connectivity.
\end{remark}

\begin{lemma}\label{lemma-type0-red}
Suppose $\rho$ is a type-$0$ representation with classical shadow $z=(z_0,z_1,z_\infty,w)\in\centV$ where $z_0\ne\pm2$. Then $\rho$ is reducible if and only if $z$ is a singular point of $\centV$.
\end{lemma}

\begin{proof}[Proof when $n$ is odd]
($\Rightarrow$)
By Lemma~\ref{lemma-red} and \eqref{eq-x1inf-a}, $z_1=z_\infty=0$. On the other hand, since $r_i=r_j=0$ for some $i\ne j$,
\begin{equation}
w+q^{4i+2}\sigma^2+q^{-4i-2}\sigma^{-2}=0=w+q^{4j+2}\sigma^2+q^{-4j-2}\sigma^{-2}.
\end{equation}
Since the order of $q^4$ is $N$, which is odd, $q^{4i+2}\ne q^{4j+2}$. This implies
\begin{equation}
q^{4i+2}\sigma^2=q^{-4j-2}\sigma^{-2}\ne q^{4j+2}\sigma^2,\qquad
q^{4(i+j+1)}\sigma^4=1.
\end{equation}
Therefore, $\sigma^{4N}=1$. Since $z_0\ne\pm2$ implies $\sigma^N\ne\pm1$, we have $\sigma^N=\pm\iunit$. Thus, $z_0=\sigma^N+\sigma^{-N}=0$, and the classical shadow has the form $(0,0,0,w)$ with $T_N(w)=2$. Finally, $q^{-4j-2}\sigma^{-2}\ne q^{4j+2}\sigma^2$ implies $q^{4j+2}\sigma^2\ne\pm1$, so $w\ne2$, and the classical shadow is a singular point.

($\Leftarrow$) The only singular points with $z_0\ne\pm2$ are $z=(0,0,0,w)$. Since $z_0=0$, after relabeling the basis elements, we can set $\sigma=\iunit$. By \eqref{eq-x1inf-a}, $z_1=z_\infty=0$ implies $\Pi_s=\Pi_t=0$. Since $T_N(w)=2$, we can write $w=q^{2k}+q^{-2k}$, with $k$ not divisible by $N$ since $w\ne2$. Therefore,
\begin{equation*}
w+q^{4i+2}\sigma^2+q^{-4i-2}\sigma^{-2}
=q^{2k}+q^{-2k}-q^{4i+2}-q^{-4i-2}
\end{equation*}
vanishes for two different $i$ corresponding to $q^{4i+2}=q^{\pm2k}$. The corresponding $r_i$ also vanish. Therefore, $\rho$ is reducible by Lemma~\ref{lemma-red}.
\end{proof}

\begin{proof}[Proof when $n$ is even]
($\Rightarrow$)
Using Lemma~\ref{lemma-red} and \eqref{eq-x1inf-b}, $z_1=z_\infty=f$. We also have $w+q^{4i+2}\sigma^2+q^{-4i-2}\sigma^{-2}=0$ since $r_i=0$. Then
\begin{align}
T_N(w)&=(-q^{4i+2}\sigma^2)^N+(-q^{-4i-2}\sigma^{-2})^N
=-\epsilon^2z_0.\\
f&=\frac{2\epsilon^2(-\epsilon^2z_0)+4}{z_0-2}=-2.
\end{align}
Thus, the classical shadow has the form $(z_0,-2,-2,w)$ with $T_N(w)=-\epsilon^2z_0$. This is a singular point of $\centV$.

($\Leftarrow$)
We essentially reverse the calculation above. Since $z_0\ne\pm2$, the singular point is of the form $(z_0,-2,-2,w)$ with $T_N(w)=-\epsilon^2z_0$. From the definitions, $f=-2$ and $R=0$. Then \eqref{eq-x1inf-b} implies $\Pi_s=\Pi_t=0$. On the other hand, $T_N(w)=-\epsilon^2z_0$ means $w=-q^{4k+2}\sigma^2-q^{-4k-2}\sigma^{-2}$ for some $k$. Therefore, $w+q^{4i+2}\sigma^2+q^{-4i-2}\sigma^{-2}$ vanishes for $i=k,k+N$, and so do the corresponding $r_i$. Therefore, the type-$0$ representation is reducible by Lemma~\ref{lemma-red}.
\end{proof}

\subsection{Exceptional points when $n$ is odd}

By the proof of Lemma~\ref{lemma-except}, if $(z_0,z_1,z_\infty,w)$ is exceptional, then $T_N(w)=-2$. Furthermore, it is singular if and only if $w\ne-2$.

The action of $H^\ast=H^1(\ptorus;\mathbb{Z}/2)$ on the skein algebra that induces sign changes on $z_m$. The exceptional points with the same $w$ are in the same orbit of this action. Thus, we can focus on any choice of sign for $z_m$.

\begin{proposition}\label{prop-except-odd}
Define $\bar{N}=\frac{N-1}{2}$. Let $V$ be an $N$-dimensional vector space with a basis $\{u_i\mid i=-\bar{N},\dotsc,\bar{N}\}$. Let $\lambda_i=q^{2i}+q^{-2i}$. For $i\ne0$, let
\begin{equation}
s_i=\frac{q^{2i-k+1}-q^{-2i+k-1}}{q^{2i}-q^{-2i}},\qquad
s'_i=\frac{q^{2i+2}-q^{-2i-2}}{q^{2i}-q^{-2i}}s_i.\qquad
\beta_i=\frac{2q^{k-1}-2q^{-k+1}}{(q^{2i}-q^{-2i})^3}.
\end{equation}
Finally, let
\begin{equation}
u^\ast_i=\begin{cases}u_{-\bar{N}},&i<-\bar{N},\\u_i,&-\bar{N}\le i\le \bar{N},\\u_{\bar{N}},&i>\bar{N}.\end{cases}
\end{equation}
Then the equations
\begin{align}
\rho(\alpha_0)u_i&=
\begin{cases}
\lambda_iu_i,&-\bar{N}\le i\le0,\\
\lambda_iu_i+u_{-i},&0<i\le\bar{N},
\end{cases}\\
\rho(\alpha_\infty)u_i&=
\begin{cases}
s_iu_{i+1}+s_{-i}u^\ast_{i-1},&-\bar{N}\le i<0,\\
(q^{k-1}+q^{-k+1})u_{-1}+(q^2-q^{-2})(q^{-k+1}-q^{k-1})u_1,&i=0,\\
\beta_iu^\ast_{-i-1}-\beta_iu_{-i+1}+s'_{-i}u_{i-1}+s'_iu^\ast_{i+1},&0<i\le\bar{N}
\end{cases}
\end{align}
define a representation $\rho:\skein_q(\ptorus)\to\End(V)$ with classical shadow $(2,2\epsilon^k,2\epsilon^{k-1},w)$ where $w=-q^{2k}-q^{-2k}$.
\end{proposition}

The construction of this representation is given in Appendix~\ref{sec-except-odd}.

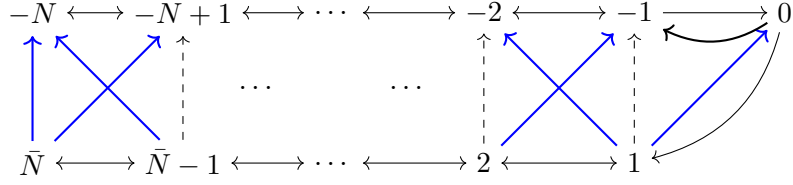
\begin{figure}
\centering
\begin{tikzpicture}[every loop/.style={}]
\node (N) at (10,1){$0$};
\node (1) at (8,1){$-1$};\node (m1) at (8,-1){$1$};
\node (2) at (6,1){$-2$};\node (m2) at (6,-1){$2$};
\node (d) at (4,1){$\dots$};\node (md) at (4,-1){$\dots$};
\path (1) edge[<->] (2) edge[->] (N);
\path (m1) edge[<->] (m2) edge[thick,blue,->] (N);
\path[bend left,->] (N) edge[thick] (1) edge (m1);
\path[dashed,->] (m1) edge (1);
\path[dashed,->] (m2) edge (2);
\path (2) edge[<->] (d) (m2) edge[<->] (md);
\node (hm1) at (2,1){$-\bar{N}+1$};
\node (hp1) at (2,-1){$\bar{N}-1$};
\node (hm) at (0,1){$-\bar{N}$};
\node (hp) at (0,-1){$\bar{N}$};
\path[<->] (hm1) edge (d) edge (hm);
\path[<->] (hp1) edge (md) edge (hp);
\path[thick,blue,->] (m1) edge (2);
\path[thick,blue,->] (m2) edge (1);
\draw (3,0)node{$\dots$} (5,0)node{$\dots$};
\path (hp1) edge[thick,blue,->] (hm) edge[dashed,->] (hm1);
\path[thick,blue,->] (hp) edge (hm1) edge (hm);
\end{tikzpicture}
\caption{The associated graph of $\rho(\alpha_\infty)$ when $n$ is odd}\label{fig-GZ-odd}
\end{figure}

\begin{lemma}\label{lemma-red-2o}
The representation above is reducible if and only if $w\ne-2$. 
\end{lemma}

\begin{proof}
Observe that the representation depends on $k\bmod{2N}$, but replacing $k$ with $k+N$ has the effect of replacing $\rho(\alpha_\infty)$ with $\epsilon\rho(\alpha_\infty)$, which does not change reducibility. Therefore, we restrict to the range $\abs{k}\le\bar{N}$. We use the notation $V_B=\Span\{u_b\mid b\in B\}$ for $B\subset\{-\bar{N},\dotsc,\bar{N}\}$.

First, we deal with the special case $w=-q^2-q^{-2}$. If $k=1$, then the coefficient of $u_1$ in $\rho(u_0)$ vanishes. This shows that $B=\{-\bar{N},\dotsc,0\}$ defines a proper and nontrivial invariant subspace, so the representation is reducible. If $k=-1$, then $s_{-1}=0$. We can take $B=\{-\bar{N},\dotsc,-1\}$ to show the representation is reducible.

Next, suppose $1<\abs{k}\le\bar{N}$, so we still have $w\ne-2$. Let
\begin{equation*}
l=\begin{cases}
\bar{N}-\frac{\abs{k}}{2},&k\text{ is even},\\
\frac{\abs{k}-1}{2},&k\text{ is odd}.
\end{cases}
\end{equation*}
Note $0<l<\bar{N}$. Then
\begin{equation*}
s_ls_{-l-1}=\frac{q^{2k}+q^{-2k}-q^{4l+2}-q^{-4l-2}}{(q^{2l}-q^{-2l})(q^{-2l-2}-q^{2l+2})}=0.
\end{equation*}
If $s_l=0$, let $B=\{-\bar{N},\dotsc,l\}$. If $s_{-l-1}=0$ instead, let $B=\{-\bar{N},\dotsc,-l-1\}$. A direct check shows that $V_B$ as defined before is a proper and nontrivial invariant subspace.

Finally, suppose $k=0$, so $w=-2$. A quick calculation shows that $s_i\ne0$ for $i\ne0,\bar{N}$ and $s'_i\ne0$ for $0<i<\bar{N}$. Now we show that the representation is irreducible. Let $U$ be a nontrivial invariant subspace. Then the argument of Lemma~\ref{lemma-red} shows that all eigenvectors $u_{-\bar{N}},\dotsc,u_0$ are in $U$. Then
\begin{equation*}
u_1=\frac{\rho(\alpha_\infty)u_0-(q+q^{-1})u_{-1}}{(q^2-q^{-2})(q-q^{-1})}
\end{equation*}
is also in $U$. The remaining basis vectors follow inductively using
\begin{equation*}
s'_iu_{i+1}=\rho(\alpha_\infty)u_i-\beta_iu_{-i-1}+\beta_iu_{-i+1}-s'_{-i}u_{i-1},\quad i=1,\dotsc,\bar{N}-1.\qedhere
\end{equation*}
\end{proof}

\begin{remark}\label{rem-graph}
Similar to Remark~\ref{rem-red}, we include an illustration of the graph associated to $\rho(\alpha_\infty)$ for visualization of the proof. 
This is shown in Figure~\ref{fig-GZ-odd} using solid lines. The vertices of the top row correspond to the eigenvectors. The horizontal straight lines are $s_i$ and $s'_i$. The thick black line $0\to-1$ is the coefficient $q^{k-1}+q^{-k-1}\ne0$. The blue lines correspond to $\beta_i$. The dashed lines indicate the relation $\rho(\alpha_0-\lambda_i)u_i=u_{-i}$ for $0<i\le\bar{N}$. There are two self-loops at $\pm\bar{N}$ omitted.
\end{remark}


\subsection{Exceptional points when $n$ is even}

After possibly a diffeomorphism of $\ptorus$, the exceptional points are
\begin{enumerate}
\item $(-2,-2,-2,w)$ with $T_N(w)=2\epsilon^2$,
\item $(-2,-2,2,w)$ with $T_N(w)=-2\epsilon^2$, and
\item $(2,2,2,w)$ with $T_N(w)=-2\epsilon^2$
\end{enumerate}
by the proof of Lemma~\ref{lemma-except}.

In the first case, $z_0=-2$ is still general, so we can construct a type-$0$ representation. Since
\begin{equation*}
f=\frac{2\epsilon^2T_N(w)+4}{z_0-2}=-2,\qquad
R=\left(\dfrac{T_N(w)+\epsilon^2z_0}{z_0-2}\right)^2=0,
\end{equation*}
\eqref{eq-m2-exist} has a unique solution $(x,y)=(0,0)$. $R=0$ also implies $r_i=0$ for some $i$. Then $r_{i+N}=0$ as well. By Lemma~\ref{lemma-red}, the representation is reducible.

Similarly, in the second case, write $w=-q^{4k+2}-q^{-4k-2}$, and take $\sigma=q^{-1}$. Since
\begin{equation*}
f=\frac{2\epsilon^2T_N(w)+4}{z_0-2}=0,\qquad
R=\left(\dfrac{T_N(w)+\epsilon^2z_0}{z_0-2}\right)^2=1,
\end{equation*}
\eqref{eq-m2-exist} has a unique solution $(x,y)=(1,1)$. This is achieved by
\begin{equation}
s_i=\frac{q^{2i-2k-1}-q^{-2i+2k+1}}{q^{2i-1}-q^{-2i+1}},\quad t_i=s_{-i},
\end{equation}
which can be verified using Lemma~\ref{lemma-Cheb-prod}. Note our choice of $\sigma$ implies $\lambda_i=\lambda_{-i+1}$. Let
\begin{equation}
u^\pm_i=v_i\pm v_{-i+1}=\pm u^\pm_{-i+1},\qquad V^\pm=\Span\{u^\pm_1,\dotsc,u^\pm_N\}.
\end{equation}
Then $u^\pm_i$ is an eigenvector of $\alpha_0$ with eigenvalue $\lambda_i$, and $\{u^\pm_i\}_{i=1}^N$ is an eigenvector basis of $V^\pm$ with distinct eigenvalues. A quick calculation shows
\begin{equation}
\rho(\alpha_\infty)u^\pm_i=s_iu^\pm_{i+1}+t_{i-1}u^\pm_{i-1},
\end{equation}
which shows that $V^\pm$ are invariant subspaces. (In fact, the argument of Lemma~\ref{lemma-red} shows that $V^\pm$ are irreducible.)

Finally, we deal with the case $(2,2,2,w)$, where we can write $w=-q^{4k+2}-q^{-4k-2}$.

\begin{proposition}\label{prop-except-even}
Let $V$ be a $2N$-dimensional vector space with a basis $\{u_i\mid i=-N,\dotsc,N-1\}$. Let $\lambda_i=q^{2i}+q^{-2i}$. For $i\ne-N,0$, let
\begin{equation}
s_i=\frac{q^{2i-2k}-q^{-2i+2k}}{q^{2i}-q^{-2i}},\qquad
s'_i=\frac{q^{2i+2}-q^{-2i-2}}{q^{2i}-q^{-2i}}s_i.\qquad
\beta_i=\frac{q^{2k}-q^{-2k}}{(q^{2i}-q^{-2i})^3}.
\end{equation}
Then the equations
\begin{align}
\rho(\alpha_0)u_i&=
\begin{cases}
\lambda_iu_i,&-N\le i\le 0,\\
\lambda_iu_i+u_{-i},&0<i<N,
\end{cases}\\
\rho(\alpha_\infty)u_i&=
\begin{cases}
(q^{2k}+q^{-2k})u_{-N+1}+(q^2-q^{-2})(q^{2k}-q^{-2k})u_{N-1},&i=-N,\\
s_iu_{i+1}+s_{-i}u_{i-1},&-N<i<0,\\
(q^{2k}+q^{-2k})u_{-1}-(q^2-q^{-2})(q^{2k}-q^{-2k})u_1,&i=0,\\
\beta_iu_{-i-1}-\beta_iu_{-i+1}+s'_{-i}u_{i-1}+s'_iu_{i+1},&0<i<N
\end{cases}
\end{align}
define a representation $\rho:\skein_q(\ptorus)\to\End(V)$ with classical shadow $(2,2,2,w)$ where $w=-q^{4k+2}-q^{-4k-2}$.
\end{proposition}

Note the formula uses a non-existent vector $u_N$, but the coefficient $s'_{N-1}$ is $0$ by definition.

The construction of this representation is given in Appendix~\ref{sec-except-even}. Figure~\ref{fig-GZ-4} shows the associated graph of $\rho(\alpha_\infty)$ using the same conventions as Figure~\ref{fig-GZ-odd} and Remark~\ref{rem-graph}. (Note $s'_{-1}=s'_{N-1}=0$.)

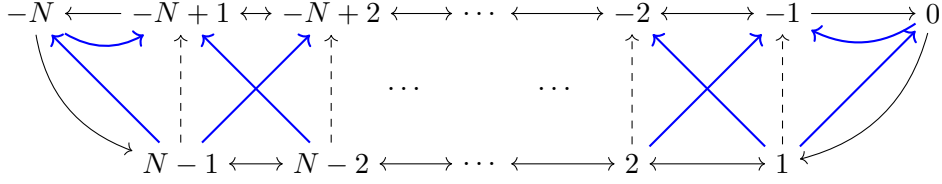
\begin{figure}
\centering
\begin{tikzpicture}[every loop/.style={}]
\node (2N) at (-1,1){$-N$};
\node (1) at (1,1){$-N+1$};\node (m1) at (1,-1){$N-1$};
\node (2) at (3,1){$-N+2$};\node (m2) at (3,-1){$N-2$};
\node (d) at (5,1){$\dots$};\node (md) at (5,-1){$\dots$};
\path (1) edge[<->] (2) edge[->] (2N);
\path (m1) edge[<->] (m2) edge[thick,blue,->] (2N);
\path (2) edge[<->] (d); \path (m2) edge[<->] (md);
\path[->,bend right] (2N) edge[thick,blue] (1) edge (m1);
\path[dashed,->] (m1) edge (1); \path[dashed,->] (m2) edge (2);
\path[thick,blue,->] (m1) edge (2);
\path[thick,blue,->] (m2) edge (1);
\node (hm1) at (7,1){$-2$};\node (hp1) at (7,-1){$2$};
\node (hm) at (9,1){$-1$};\node (hp) at (9,-1){$1$};
\path[<->] (hm1) edge (d) edge (hm);
\path[<->] (hp1) edge (md) edge (hp);
\draw (4,0)node{$\dots$} (6,0)node{$\dots$};
\node (N) at (11,1){$0$};
\path[->] (hp1) edge[thick,blue] (hm) edge[dashed] (hm1);
\path[->] (hp) edge[thick,blue] (hm1) edge[dashed] (hm);
\path[->,bend left] (N) edge[thick,blue] (hm) edge (hp);
\path[<-] (N) edge[thick,blue] (hp) edge (hm);
\end{tikzpicture}
\caption{The associated graph of $\rho(\alpha_\infty)$ when $n$ is even}\label{fig-GZ-4}
\end{figure}

\begin{lemma}\label{lemma-red-2e}
The representation above is reducible if and only if $w\ne2$.
\end{lemma}

\begin{proof}
Note replacing $k$ by $k+N$ changes $\rho(\alpha_\infty)$ to $-\rho(\alpha_\infty)$. 
Thus, we can restrict to the range $0\le k\le N-1$. We use the notation $V_B=\Span\{u_b\mid b\in B\}$ for $B\subset\{-N\dotsc,N-1\}$. The proof is similar to Lemma~\ref{lemma-red-2o}.

First, suppose $k=0$. Then $q^{2k}-q^{-2k}=0$. This implies $B=\{-N,\dotsc,0\}$ defines a proper and nontrivial invariant subspace. Similarly, if $k=N-1$, then $s_{-1}=s_{N-1}=0$. We can take $B=\{-N+1,\dotsc,-1\}$.

Next, suppose $1\le k\le N-2$ and $k\ne\frac{N-1}{2}$. In this case, $s_{-N+k}=s_{k}=0$ and $s'_k=s'_{-N+k}=0$. If $k<\frac{N-1}{2}$, set $B=\{-N,\dotsc,0,1,\dotsc,k\}\cup\{N-k,\dotsc,N-1\}$. Otherwise, $k>\frac{N-1}{2}$, and set $B=\{-k,\dotsc,-N+k\}$. A direct check shows that $V_B$ is a proper and nontrivial invariant subspace.

In both cases above, $w\ne2$. The remaining case is when $N$ is odd $k=\frac{N-1}{2}$. Then $w=-q^{4k+2}-q^{-4k-2}=2$. It is easy to check that $s_i\ne0$ if $i\ne-N,0,-N+k,k$, $s'_i\ne0$ if $i\ne-N,0,-N+k,k,-1,N-1$, and $q^{2k}-q^{-2k}\ne0$, so $\beta_i\ne0$ if $i\ne-N,0$. Let $U$ be a nontrivial invariant subspace. Starting with an eigenvector of $\rho(\alpha_\infty)$ in $U$, the argument of Lemma~\ref{lemma-red} shows that at least one of $\{u_{-N},\dotsc,u_{-k-1}\}$ or $\{u_{-k},\dotsc,u_0\}$ is in $U$. Take the first possibility for example. The argument of Lemma~\ref{lemma-red-2o} shows that $U$ also contains $u_{N-1},\dotsc,u_{N-k}=u_{k+1}$. Then use
\begin{equation*}
\beta_{k+1}u_{-k}=\beta_{k+1}u_{-k-2}+s'_{k+1}u_{k+2}-\rho(\alpha_\infty)u_{k+1}
\end{equation*}
to show that $u_{-k}\in U$. Then apply the argument of Lemma~\ref{lemma-red} again to show that the rest of the eigenvectors are in $U$, and we can use the inductive argument of Lemma~\ref{lemma-red-2o} again to obtain the remaining vectors. Thus, the representation is irreducible.
\end{proof}

\section{Skein algebra of the four-punctured sphere}

The next simplest surface is the four-punctured sphere $\hsphere$. The skein algebra $\skein_q(\hsphere)$ has a presentation similar to $\ptorus$. However, the bigger homology of $\hsphere$ makes the algebra more complicated. We also determine the singularities and exceptional points, just like the one-punctured torus case.

Many results in this section and the next one are results of brute force calculation. We often omit the details of the proofs.

\subsection{Presentation}

There is a parameterization of simple closed curves on $\hsphere$ by slopes $m=a/b$. Represent $\hsphere$ as in Figure~\ref{fig-puncture-a-b}. Choose the curves $\alpha_0,\alpha_\infty$, and number the punctures as in the figure. The curve surrounding puncture $i$ is denoted by $p_i$. To obtain the slope-$(a/b)$ curve $\alpha_{a/b}$ on $\hsphere$, take $\abs{a}$ parallel copies of $\alpha_\infty$ and $\abs{b}$ parallel copies of $\alpha_0$, and resolve each intersection such that one would turn left from $\alpha_0$ to $\alpha_\infty$ if $ab>0$ and turn right if $ab<0$. The curve $\alpha_1$ is demonstrated in Figure~\ref{fig-curve-b1-on-S0}.

\begin{figure}
\centering
\begin{subfigure}[t]{0.4\linewidth}
\centering
\begin{tikzpicture}
\draw[fill=gray!20!white,dotted] (0,0)ellipse[x radius=2,y radius=1];
\fill[black] (-1,0)circle[radius=0.05]
	(0,0)circle[radius=0.05]
	(1,0)circle[radius=0.05];
\draw (-1,0)node[below]{1} (0,0)node[below]{2} (1,0)node[below]{3};
\draw (-2,-0.5)node{4};
\draw (-0.5,0)ellipse[x radius=0.9,y radius=0.6];
\draw (0.5,0)ellipse[x radius=0.9,y radius=0.6];
\draw (-1.4,0)node[left]{$\alpha_0$} (1.4,0)node[right]{$\alpha_\infty$};
\end{tikzpicture}
\subcaption{$\alpha_0$, $\alpha_\infty$, and puncture numbering}
\label{fig-puncture-a-b}
\end{subfigure}
\quad
\begin{subfigure}[t]{0.4\linewidth}
\centering
\begin{tikzpicture}
\draw[fill=gray!20!white,dotted] (0,0)ellipse[x radius=2,y radius=1];
\fill[black] (-1,0)circle[radius=0.05]
	(0,0)circle[radius=0.05]
	(1,0)circle[radius=0.05];
\draw (0.86,-0.14) to[closed,
	curve through={(1.2,-0.2)..(1.14,0.14)..(0.5,0.7)
	..(0,0.8)..(-0.5,0.7)..(-1.14,0.14)..(-1.2,-0.2)
	..(-0.86,-0.14)}] (0,0.4);
\end{tikzpicture}
\subcaption{The curve $\alpha_1$}
\label{fig-curve-b1-on-S0}
\end{subfigure}
\caption{Curves on $\hsphere$}
\label{fig-curves-on-S04}
\end{figure}

$\skein_q(\hsphere)$ also has a presentation given in \cite{BP}. Define the following elements of $\skein_q(\hsphere)$
\begin{gather}
\gamma=p_1p_2p_3p_4+p_1^2+p_2^2+p_3^2+p_4^2,\\
c_0=p_1p_2+p_3p_4,\quad
c_1=p_1p_3+p_2p_4,\quad
c_\infty=p_1p_4+p_2p_3.
\end{gather}
More generally, for a slope $m$, if $\alpha_m$ separates the peripheral curves in the same pattern as $\alpha_0$, $\alpha_1$, or $\alpha_\infty$, then $c_m$ is defined as the corresponding quantity above.

Define a function
\begin{equation}
H(x,y,z;a,b,c,g)=x^2+y^2+z^2-xzy+ax+by+cz+g,
\end{equation}
where the variables $x,y,z$ are not necessarily commutative. Then $\skein_q(\hsphere)$ is generated over $\cx[p_1,p_2,p_3,p_4]$ by $\alpha_0,\alpha_1,\alpha_\infty$ with relations
\begin{align}
H(q^2\alpha_0,q^2\alpha_1,q^{-2}\alpha_\infty&;c_0,c_1,c_\infty,\gamma)=(q^2+q^{-2})^2,\label{eq-relG-gen}\\
q^2\alpha_0\alpha_\infty-q^{-2}\alpha_\infty\alpha_0&=(q^4-q^{-4})\alpha_1+(q^2-q^{-2})c_1,\label{eq-rel04-1}\\
q^2\alpha_1\alpha_0-q^{-2}\alpha_0\alpha_1&=(q^4-q^{-4})\alpha_\infty+(q^2-q^{-2})c_\infty,\label{eq-rel04-inf}\\
q^2\alpha_\infty\alpha_1-q^{-2}\alpha_1\alpha_\infty&=(q^4-q^{-4})\alpha_0+(q^2-q^{-2})c_0.\label{eq-rel04-0}
\end{align}
If $q$ is not an 8th root of unity, then $\alpha_1$ is redundant.


\subsection{Center at a root of unity}

Now let $q$ be a root of unity. By \eqref{eq-PIdeg}, the PI degree of $\skein_q(\hsphere)$ is $D=N=\order(q^4)$.

To avoid confusion, elements of $\skein_\epsilon(\hsphere)$ are denoted with a hat. Define $A_m=\epsilon^2T_N(\alpha_m)=\epsilon^2\Ch(\hat{\alpha}_m)\in Z$. Note the inclusion of the sign $\epsilon^2=\pm1$ in this case (compare Theorem~\ref{thm-even-iso}). In addition, let $P_i=\Ch(\hat{p}_i)=T_N(p_i)$. Then $C_m=\Ch(\hat{c}_m)$ and $\Gamma=\Ch(\hat\gamma)$ are given by similar expressions as $c_m$ and $\gamma$ with $p_i$ replaced by $P_i$.

Every diagram is even on $\hsphere$, so for all roots of unity, $X=\skein_\epsilon(\hsphere)$. Since $\epsilon^2=\pm1$, \eqref{eq-rel04-1}--\eqref{eq-rel04-0} reduces to commutativity for $X$, so only \eqref{eq-relG-gen} is necessary. The presentation of the center $Z$ can then be obtained by applying $\Ch$ to \eqref{eq-relG-gen}. The generators are $A_0,A_1,A_\infty,p_1,p_2,p_3,p_4$, and there is one relation
\begin{equation}\label{eq-relG}
H(A_0,A_1,A_\infty;C_0,C_1,C_\infty,\Gamma)=(\epsilon^2+\epsilon^{-2})^2=4.
\end{equation}

\subsection{Exceptional points}

Like the one-punctured torus case, we say $z\in\centV$ or $\charV$ is \term{exceptional} if $A_m(z)=\pm2$ for all slopes $m$.

\begin{lemma}\label{lemma-except-04}
There are finitely many exceptional points in $\charV$ (and in $\centV$ as a result). They are the characters of the following representations $r:\pi_1(\hsphere)\to SL_2$.
\begin{enumerate}
\item The image of $r$ is central $\{\pm I\}$.
\item The image of $r$ is a cyclic subgroup of $SL_2$ of order $4$.
\item Up to permutations of the punctures, $r$ has the following form
\begin{equation}\label{eq-allpm2}
r(p_1)=\pm I,\quad r(p_2)=\pm\begin{pmatrix}1&2\\0&1\end{pmatrix},\quad
r(p_3)=\pm\begin{pmatrix}1&0\\-2&1\end{pmatrix}.
\end{equation}
\end{enumerate}
\end{lemma}

Here, $p_1,\dotsc,p_4\in\pi_1(\hsphere)$ are chosen such that $p_1p_2=\alpha_0$, $p_2p_3=\alpha_\infty$, $p_1p_3=\alpha_1$, $p_1p_2p_3p_4=1$.

\begin{proof}
By direct calculation using the defining relations of $\skein_\epsilon(\surface)$,
\begin{equation}
\hat\alpha_0\hat\alpha_\infty=\epsilon^2\hat\alpha_1+\epsilon^2\hat\alpha_{-1}+\hat{c}_1.
\end{equation}
Applying $\Ch$, we get $A_0A_\infty=A_1+A_{-1}+C_1$. Then using diffeomorphisms of the surface, if $\abs{ad-bc}=1$,
\begin{equation}
A_{\frac{a}{b}}A_{\frac{c}{d}}=A_{\frac{a+c}{b+d}}+A_{\frac{a-c}{b-d}}+c_{\frac{a+c}{b+d}}.
\end{equation}
By choosing $(a/b,c/d)=(0,1),(0,-1),(0,1/2),(0,\infty),(1,\infty),(2,\infty)$ and requiring all $A_m$ appeared to be $\pm2$, we obtain finitely many possibilities for $(C_0,C_1,C_\infty,\Gamma)$. This can be done easily with an exhaustive search using a computer program. Then we can solve for $\vec{W}$ to obtain a finite list of candidates of exceptional points.

To reduce the workload, we use symmetries of $\charV$. Diffeomorphisms of $\hsphere$ acts on $\pi_1(\hsphere)$, which induces an action on $\charV$. This action can permute punctures, and every even permutation can be realized as a diffeomorphism that permutes the free homotopy classes of $\alpha_0,\alpha_1,\alpha_\infty$. In addition, a cohomology class $h\in H^\ast$ acts on $r:\pi_1(\hsphere)\to SL_2$ by $(h\cdot r)(\alpha)=(-1)^{h(\alpha)}r(\alpha)$. This action changes the signs of the coordinates. Up to these symmetries, the candidate exceptional points are $z=(z_0,z_1,z_\infty,\vec{W})\in\charV$ with
\begin{enumerate}
\item $z_0=z_1=z_\infty=-2$, $W_1=W_2=W_3=W_4=-2$.
\item $z_0=z_1=z_\infty=2$, $W_1=W_2=W_3=W_4=0$,
\item $z_0=z_1=z_\infty=2$, $W_1=2$, $W_2=W_3=W_4=-2$.
\end{enumerate}
It is easy to verify that they are the characters of the 3 possibilities in the lemma for a suitable choice of signs.

We need to show that these points are actually exceptional. The first two are obvious, so we focus on the last one. We fix the signs of $r$ to be $r(p_1)=-I$ and $r(p_2),r(p_3)$ take the plus sign in \eqref{eq-allpm2}. Recall there is a braid group $B_3$ embedded in the mapping classes group of $\hsphere$. Using the standard generators of the braid group, it is easy to check by induction that for any $b\in B_3$, one of $r(b(p_i)),i=1,2,3$ is $-I$, the other two have trace $2$, and the product of any two has trace $-2$. This implies $A_m(z)=2$ for all slopes $m$.
\end{proof}

\subsection{Singularities}

In this section we describe the singularities of $\centV$ and the slices $\centV_{\vec{w}}$. For $\vec{w}\in\cx^4$, recall the notation $W_i=T_N(w_i)$ and $\vec{W}=(W_1,\dotsc,W_4)$.

\begin{theorem}[{\cite[Theorem~3.7]{CL}}]\label{thm-sing-sl}
A point is singular in the slice $\centV_{\vec{w}}\cong\charV_{\vec{W}}$ if and only if it is a character of one of the following.
\begin{enumerate}
\item A representation $r:\pi_1(\hsphere)\to SL_2$ with $r(p_i)=\pm I$ for some $i\in\{1,2,3,4\}$.
\item A reducible representation $\pi_1(\hsphere)\to SL_2$.
\end{enumerate}
\end{theorem}

Define
\begin{equation}
\kappa(a,b,c)=a^2+b^2+c^2+abc-4.
\end{equation}

\begin{lemma}\label{lemma-sing-defeq}
Let $z=(z_0,z_1,z_\infty)\in\centV_{\vec{w}}\cong\charV_{\vec{W}}$.
\begin{enumerate}
\item For $i\in\{1,2,3,4\}$, $z$ is the character of $r:\pi_1(\hsphere)\to SL_2$ such that $r(p_i)=\pm I$ if and only if $W_i=\pm2$ and $z_m=-\frac{W_i}{2}W_{j_m}$ for $m\in\{0,1,\infty\}$, where $j_m$ is the puncture on the same side of $\alpha_m$ as $i$.
\item $z$ is the character of a reducible $r:\pi_1(\hsphere)\to SL_2$ if and only if
\begin{equation}
\left\{
\begin{alignedat}{3}
&\kappa(W_1,W_2,z_0)&&=\kappa(W_3,W_4,z_0)&&=0,\\
&\kappa(W_1,W_3,z_1)&&=\kappa(W_2,W_4,z_1)&&=0,\\
&\kappa(W_1,W_4,z_\infty)&&=\kappa(W_2,W_3,z_\infty)&&=0.
\end{alignedat}
\right.
\end{equation}
\end{enumerate}
\end{lemma}

\begin{proof}
(1) $\Rightarrow$ is clear. For $\Leftarrow$, without loss of generality, assume $i=1$. Up to conjugation, we can write $r(p_1)=s\begin{pmatrix}1&k\\0&1\end{pmatrix}$ where $s=-\frac{W_1}{2}$. Let $r(p_2)=\begin{pmatrix}a&b\\c&d\end{pmatrix}$. Then $r(\alpha_0)=r(p_1p_2)=s\begin{pmatrix}a+kc&b+kd\\c&d\end{pmatrix}$, so $z_0=-s(a+d+kc)=-s(-W_2+kc)$. By assumption, $z_0=-\frac{W_1}{2}W_{j_0}=sW_2$, so $kc=0$. If $k=0$, then $r$ is the representation we want. Otherwise, $c=0$. The same argument shows that $r(p_j)$ is upper triangular for all $j$. A standard calculation shows that conjugates of $r$ have a limit point where $p_1$ is sent to $\pm I$.

(2) See \cite{BG0}. Note our identification of $\charV$ with the character variety has a minus sign from \eqref{eq-charV-tr}.
\end{proof}

\begin{theorem}\label{thm-sing}
Let $z\in\centV_{\vec{w}}\subset\centV$. Then $z$ is singular in $\centV$ if and only if one of the following holds.
\begin{enumerate}
\item There exists $i\in\{1,2,3,4\}$ such that $z$ is a character of a representation $r:\pi_1(\hsphere)\to SL_2$ with $r(p_i)=\pm I$ and $w_i\ne\pm2$.
\item $z$ is a character of a reducible representation $r:\pi_1(\hsphere)\to SL_2$.
\end{enumerate}
\end{theorem}

For the first type of singularities, $T'_N(w_i)=0$. We call these singularities \term{ramified}. Their images in $\charV$ are not singular in $\charV$, although they are singular in the slice $\charV_{\vec{W}}$. The second type of singularities is called \term{reducible} for the obvious reason.

\begin{lemma}\label{lemma-sing}
Suppose $z=(z_0,z_1,z_\infty,\vec{w})\in\centV$ satisfies $z_0\ne\pm2$. If $z$ is singular in $\centV$, then
\begin{equation}\label{eq-sing-z1inf}
z_1=\frac{2C_1+C_\infty z_0}{z_0^2-4},\qquad z_\infty=\frac{2C_\infty+C_1 z_0}{z_0^2-4}.
\end{equation}
In addition, one of the following holds.
\begin{enumerate}
\item There exists an $i\in\{1,2,3,4\}$ such that $T'_N(w_i)=0$ and $\kappa(W_i,W_j,z_0)=0$, where $j$ is the puncture on the same side of $\alpha_0$ as $i$.
\item $\kappa(W_1,W_2,z_0)=\kappa(W_3,W_4,z_0)=0$.
\end{enumerate}
Conversely, a point satisfying the conditions above is singular in $\centV$.
The two types of singularities above correspond to the two types in Theorem~\ref{thm-sing}.
\end{lemma}

\begin{proof}
The singularities are critical points of $H$. Start with the partial derivatives
\begin{equation}\label{eq-H-1inf}
\frac{\partial H}{\partial z_1}=2z_1-z_0z_\infty+C_1=0,\qquad
\frac{\partial H}{\partial z_\infty}=2z_\infty-z_0z_1+C_\infty=0.
\end{equation}
These are linear in $z_1,z_\infty$, and the coefficient matrix is nondegenerate when $z_0\ne\pm2$. The solution is given by \eqref{eq-sing-z1inf}. Now consider the derivatives
\begin{equation}
\frac{\partial H}{\partial w_i}=\frac{\partial H}{\partial W_i}T'_N(w_i)=0.
\end{equation}

Case 1: Suppose $T'_N(w_k)=0$ for some $k$. Without loss of generality, assume $k=1$. Then $T'_N(w_1)=0$ implies $W_1=\pm2$. If $\kappa(W_1,W_2,z_0)=0$, then the point is in (1) with $i=1$. If $\kappa(W_1,W_2,z_0)\ne0$, we reduce the system (using Gr\"obner basis for example)
\begin{equation}
H=4,\quad \frac{\partial H}{\partial z_0}=0,\quad \text{Equation \eqref{eq-sing-z1inf}},\quad
W_1^2=4,\quad \kappa(W_1,W_2,z_0)\ne0
\end{equation}
to obtain $(W_3^2-4)(W_4^2-4)=0$ and $2z_0+W_3W_4=0$, which implies $\kappa(W_3,W_4,z_0)=0$. Exactly one of $W_3,W_4$ is $\pm2$ since $z_0\ne\pm2$. The two cases are similar, so we consider $W_3=\pm2$. If $T'_N(w_3)=0$, then the point is in (1) with $i=3$. Otherwise, we compute
\begin{equation}
\frac{\partial H}{\partial W_3}=W_4z_0+W_1z_1+W_2z_\infty+2W_3+W_1W_2W_4
=-\frac{1}{2}W_3(W_2^2-2\epsilon'W_2W_4+W_4^2)=0,
\end{equation}
where $\epsilon'=\frac{W_1W_3}{4}=\pm1$. This shows $W_1=\epsilon'W_3$, $W_2=\epsilon'W_4$. Therefore, $\kappa(W_1,W_2,z_0)=\kappa(W_3,W_4,z_0)=0$, so the point is in (2).

Case 2: $T'_N(w_i)\ne0$ for all $i$. Then $\frac{\partial H}{\partial W_i}=0$ for all $i$. A standard application of Gr\"obner basis shows that the point is in (2).

The converse direction is a direct check, and the type match up follows from Lemma~\ref{lemma-sing-defeq}.
\end{proof}

\begin{proof}[Proof of Theorem~\ref{thm-sing}]
If $z$ is not exceptional, then we can apply a diffeomorphism to assume $z_0\ne\pm2$. This case is covered by Lemma~\ref{lemma-sing}.

Note the number of exceptional points in $\charV$ is less than $2^3\times3^4$ since each $z_0,z_1,z_\infty$ can be $\pm2$, and each $W_i$ can be $\pm2$ or $0$. The exceptional points of $\centV$ can be grouped by their projections and subdivided according to whether each $T'_N(w_i)=0$. This is a partition into less than $2^3\times3^4\times2^4$ cases, which can be checked easily by a computer program. 
\end{proof}

\section{Representations for the four-punctured sphere}

In this section, we use the same strategy as the one-punctured torus case to construct representations of $\skein_q(\hsphere)$ and determine their reducibility. We assume that $q$ is not an 8th root of unity. In this case, $N\ge 3$.

Recall the two types of singularities introduced in Theorem~\ref{thm-sing}. Let $\centV_\ram,\centV_\red$ be the sets of ramified and reducible singular points, respectively. We can parameterize $\centV_\red$ as follows. Let 
\begin{equation}
B_\omega=\{\vec{\eta}\in(\cx^\times)^4\mid \eta_1\eta_2\eta_3\eta_4=\omega\},\qquad
B=\bigcup_{\omega^N=1}B_\omega.
\end{equation}
Define
\begin{equation}
h:B\to\centV_\red,\qquad h(\vec{\eta})=(z_0,z_1,z_\infty,\vec{w})	
\end{equation}
where $w_i=\eta_i+\eta_i^{-1}$, and $z_m=-(\eta_1\eta_{j_m})^N-(\eta_1\eta_{j_m})^{-N}$ for $m\in\{0,1,\infty\}$, where $j_m$ is the puncture on the same side of $\alpha_m$ as $1$. The image is in $\centV_\red$ since it is the character of $r:\pi_1(\hsphere)\to SL_2$ given by $r(p_i)=-\begin{pmatrix}\eta_i^N&0\\0&\eta_i^{-N}\end{pmatrix}$. It is clear that $h$ is surjective.

Let
\begin{equation}
\centV'_\red=\bigcup_{\substack{\omega^N=1\\ \omega\ne1}}h(B_\omega)\subset\centV_\red,\qquad
\azuV^c=\centV_\ram\cup\centV'_\red.
\end{equation}
$\azuV^c$ is a proper subset of the singular locus. The missing component $h(B_1)$ is, in a sense, the set of characters of $N$-th roots of reducible representations $\pi_1(\hsphere)\to SL_2$.

\begin{theorem}\label{thm-main-hsphere}
Suppose $q$ is a root of unity. For every $z\in\centV$, there exists an $N$-dimensional representation $\rho:\skein_q(\hsphere)\to\End(V)$ with classical shadow $z$, and $\rho$ is reducible if and only if $z\in\azuV^c$.

The Azumaya locus of $\skein_q(\hsphere)$ is the complement of $\azuV^c$. The fully Azumaya locus of $\skein_q(\hsphere)$ is the union of smooth loci of the slices of $\charV$.
\end{theorem}

\begin{proof}
The construction of representations is given in Propositions~\ref{prop-rep04} and \ref{prop-gen-exist04} when $A_0(z)\ne\pm2$. Then all non-exceptional points are covered by applying diffeomorphisms. The exceptional points are covered by \ref{prop-except04}. The reducibility is given in Lemmas~\ref{lemma-type0-red04} and \ref{lemma-red04-2}.

The Azumaya locus follow directly from the first half of the theorem, and the fully Azumaya locus follows by comparing Theorems~\ref{thm-sing-sl} and \ref{thm-sing}.
\end{proof}

\subsection{Type-$0$ representations}

We look for $N$-dimensional representations of the sliced skein algebra $\skein_{q,\vec{w}}(\hsphere)$ where $A_0=\epsilon^2T_N(\alpha_0)$ acts as a scalar $z_0\in\cx$. 

Assume $z_0\ne\pm2$. Write $\epsilon^2z_0=\sigma^N+\sigma^{-N}$. Then $\sigma^N\ne\pm1$. Define
\begin{align}
\lambda_i&=q^{4i}\sigma+q^{-4i}\sigma^{-1},&
\lambda'_i&=q^{4i+2}\sigma+q^{-4i-2}\sigma^{-1},\\
\hat{\lambda}_i&=q^{4i}\sigma-q^{-4i}\sigma^{-1},&
\hat{\lambda}'_i&=q^{4i+2}\sigma-q^{-4i-2}\sigma^{-1}.
\end{align}
Note $\hat{\lambda}_i$ and $\hat{\lambda}'_i$ are nonzero.
\begin{align}
r_i&=r_i(\sigma,\vec{w}):=\frac{\kappa(w_1,w_2,\lambda'_i)\kappa(w_3,w_4,\lambda'_i)}{\hat{\lambda}_i\hat{\lambda}_{i+1}(\hat{\lambda}'_i)^2},\label{eq-ri04}\\
E^0_{\sigma,\vec{w}}&=\{(s_1,\dotsc,s_N,t_1,\dotsc,t_N)\in\cx^{2N}\mid s_it_i=r_i(\sigma,\vec{w})\}.
\end{align}

\begin{proposition}[Generalization of {\cite[Lemmas~18--23]{Tak}}]\label{prop-rep04}
Let $V$ be an $N$-dimensional vector space with a basis $\{v_i\mid i\in\ints/N\}$. Suppose $\epsilon^2z_0=\sigma^N+\sigma^{-N}\ne\pm2$, $\vec{w}\in\cx^4$, and $(\vec{s},\vec{t})\in E^0_{\sigma,\vec{w}}$. Let
\begin{equation}\label{eq-d}
d_i=\frac{\lambda_ic_1+(q^2+q^{-2})c_\infty}{\hat{\lambda}'_i\hat{\lambda}'_{i-1}},\qquad
d'_i=\frac{\lambda_ic_\infty+(q^2+q^{-2})c_1}{\hat{\lambda}'_i\hat{\lambda}'_{i-1}}.
\end{equation}
Then the following equations define a representation $\rho:\skein_{q,\vec{w}}(\hsphere)\to\End(V)$.
\begin{align}
\rho(\alpha_0)v_i&=\lambda_iv_i,\label{eq-rep04-0}\\
\rho(\alpha_1)v_i&=q^{4i+2}\sigma s_iv_{i+1}+d'_iv_i+q^{-4i+2}\sigma^{-1} t_{i-1}v_{i-1},\label{eq-rep04-1}\\
\rho(\alpha_\infty)v_i&=s_iv_{i+1}+d_iv_i+t_{i-1}v_{i-1}.\label{eq-rep04-inf}
\end{align}
In addition, $\rho(A_0)=z_0$.
\end{proposition}

As before, this representation will be called \term{type-$0$}.

\begin{proof}
This is a very long calculation. \eqref{eq-rel04-1} and \eqref{eq-rel04-inf} are satisfied by definition \eqref{eq-d} and \eqref{eq-rep04-0}--\eqref{eq-rep04-inf}, and \eqref{eq-relG-gen} and \eqref{eq-rel04-0} uses \eqref{eq-ri04}. The details are omitted.
\end{proof}

\begin{lemma}
We have
\begin{equation}\label{eq-kprod}
\prod_{k=1}^N\kappa(w_i,w_j,\lambda'_k)=\kappa(W_i,W_j,z_0).
\end{equation}
\end{lemma}

\begin{proof}
Start with a calculation similar to Lemma~\ref{lemma-STimg} to get
\begin{equation}\label{eq-lam-prime}
\prod_{i=1}^N(x+\lambda'_i)=T_N(x)-(-1)^Nq^{2N}(\sigma^N+\sigma^{-N})=T_N(x)+z_0,
\end{equation}
Then we use the identity
\begin{equation}\label{eq-kfactor}
\kappa(a+a^{-1},b+b^{-1},x)=(x+ab+a^{-1}b^{-1})(x+ab^{-1}+a^{-1}b).
\end{equation}
Writing $w_i=\eta_i+\eta_i^{-1}$, the product of the first factor of $\kappa(w_1,w_2,\lambda'_i)$ becomes
\begin{equation*}
\prod_{i=1}^N(\lambda'_i+\eta_1\eta_2+\eta_1^{-1}\eta_2^{-1})
=T_N(\eta_1\eta_2+\eta_1^{-1}\eta_2^{-1})+z_0
=\eta_1^N\eta_2^N+\eta_1^{-N}\eta_2^{-N}+z_0.
\end{equation*}
After repeating the calculation for the second factor, we get
\begin{equation*}
\prod_{i=1}^N\kappa(w_1,w_2,\lambda'_i)=(\eta_1^N\eta_2^N+\eta_1^{-N}\eta_2^{-N}+z_0)(\eta_1^N\eta_2^{-N}+\eta_1^{-N}\eta_2^N+z_0)=\kappa(W_1,W_2,z_0).\qedhere
\end{equation*}
\end{proof}

Define
\begin{gather}
R=R(z_0,\vec{w}):=\frac{\kappa(W_1,W_2,z_0)\kappa(W_3,W_4,z_0)}{(z_0^2-4)^2}\\
\Pi_s(\vec{s},\vec{t})=\prod_{i=1}^Ds_i,\qquad
\Pi_t(\vec{s},\vec{t})=\prod_{i=1}^Dt_i.
\end{gather}

\begin{lemma}[Compare {\cite[Lemma~22]{Tak}}]\label{lemma-STimg04}
Suppose $z_0\ne\pm2$. Then the image of the map $(\Pi_s,\Pi_t):E^0_{\sigma,\vec{w}}\to\cx^2$ is the curve $\{(x,y)\in\cx^2\mid xy=R(z_0,\vec{w})\}$.
\end{lemma}

\begin{proof}
The proof is similar to Lemma~\ref{lemma-STimg}. First we show the image is contained in the curve $xy=R$. We need
\begin{equation*}
R=\prod_{i=1}^Nr_i
=\prod_{i=1}^N\frac{\kappa(w_1,w_2,\lambda'_i)\kappa(w_3,w_4,\lambda'_i)}{\hat{\lambda}_i\hat{\lambda}_{i+1}(\hat{\lambda}'_i)^2}.
\end{equation*}
The numerators match using \eqref{eq-kprod}. For the denominator, we have
\begin{equation}\label{eq-lambda-prod04}
\begin{aligned}
\prod_{i=1}^N(\hat{\lambda}_i\hat{\lambda}'_i)
&=\begin{cases}
(\sigma^N-\sigma^{-N})(q^{2N}\sigma^N-q^{-2N}\sigma^{-N}),&N\text{ is odd},\\
(2-\sigma^N-\sigma^{-N})(2-q^{2N}\sigma^N-q^{-2N}\sigma^{-N}),&N\text{ is even}
\end{cases}\\
&=q^{2N}(\sigma^N-\sigma^{-N})^2=q^{2N}(z_0^2-4)
\end{aligned}
\end{equation}
by Lemma~\ref{lemma-Cheb-prod}. Therefore, the denominators also match.

The opposite direction is exactly the same as Lemma~\ref{lemma-STimg}.
\end{proof}

\begin{lemma}[Generalization of {\cite[Lemma~24]{Tak}}]\label{lemma-x1inf-04}
In a type-$0$ representation $\rho$ (with $z_0\ne\pm2$), 
\begin{equation}\label{eq-x1inf-04}
\rho(A_1)=\sigma^N\Pi_s+\sigma^{-N}\Pi_t+f_1,\qquad
\rho(A_\infty)=\epsilon^2\Pi_s+\epsilon^2\Pi_t+f_\infty
\end{equation}
are scalars, where
\begin{equation}\label{eq-f1inf-04}
f_1=\frac{2C_1+C_\infty z_0}{z_0^2-4},\qquad
f_\infty=\frac{2C_\infty+C_1 z_0}{z_0^2-4}.
\end{equation}
Therefore, the representation has a classical shadow.
\end{lemma}

\begin{remark}
Note that the expressions of $f_1,f_\infty$ appears in the description of singular points \eqref{eq-sing-z1inf}.
\end{remark}

\begin{proof}
The proof is similar to Lemma~\ref{lemma-x1inf}. When $z_0\ne\pm2$, the eigenvalues $\lambda_i$ of $\rho(\alpha_0)$ are distinct, which implies $A_1,A_\infty$ are diagonal. To show they are scalars, let
\begin{equation*}
\rho(A_1)v_i=z_1(i)v_i,\qquad \rho(A_\infty)v_i=z_\infty(i)v_i,
\end{equation*}
and we will show that $z_1(i),z_\infty(i)$ are given by the formulas above, which are independent of $i$.

By the same argument as Lemma~\ref{lemma-x1inf}, $z_1=z_1(i)$ and $z_\infty=z_\infty(i)$ has the form in \eqref{eq-x1inf-04}. Now apply \eqref{eq-relG} to the vector $v_i$. Using \eqref{eq-x1inf-04} and $\Pi_s\Pi_t=R$, The result is
\begin{align*}
0&=\left((\sigma^N-\sigma^{-N})f_1+\epsilon^2(1-\sigma^{2N})f_\infty+\sigma^N C_1+\epsilon^2 C_\infty\right)\Pi_s+\\
&\quad+\left((\sigma^{-N}-\sigma^N)f_1+\epsilon^2(1-\sigma^{-2N})f_\infty+\sigma^{-N} C_1+\epsilon^2 C_\infty\right)\Pi_t+\\
&\quad+(\text{terms constant on }E^0_{\sigma,\vec{w}}).
\end{align*}
The coefficients of $\Pi_s$ and $\Pi_t$ are zero for the same reason as Lemma~\ref{lemma-x1inf}. Solving for $f_1$ and $f_\infty$ gives us the form in \eqref{eq-f1inf-04}.
\end{proof}

\begin{proposition}\label{prop-gen-exist04}
There exists a type-$0$ representation for every point $(z_0,z_1,z_\infty,\vec{w})\in\centV$ with $z_0\ne\pm2$.
\end{proposition}

The proof is the same as the one-punctured torus case Proposition~\ref{prop-gen-exist}.

\subsection{Reducibility of type-$0$ representations}

The analog of Lemma~\ref{lemma-red} holds with the same proof.

\begin{lemma}\label{lemma-red04}
Suppose $z_0\ne\pm2$. A type-$0$ representation is reducible if and only if $\Pi_s=\Pi_t=0$ and $r_i=r_j=0$ for some $i\ne j$.
\end{lemma}

Recall the set $\azuV^c$ defined in the beginning of the section. 

\begin{lemma}\label{lemma-type0-red04}
Suppose $z=(z_0,z_1,z_\infty,\vec{w})\in\centV$ satisfies $z_0\ne\pm2$. Then a type-$0$ representation with classical shadow $z$ is reducible if and only if $z\in\azuV^c$.
\end{lemma}

\begin{proof}[Proof of $\Rightarrow$]
By Lemma~\ref{lemma-red04} and \eqref{eq-x1inf-04}, in a reducible type-$0$ representation, $z_1=f_1$, $z_\infty=f_\infty$. We also have $r_i=r_j=0$ for some $i\ne j$. For convenience, we cyclically permute the basis so that $j=0$.

Case 1: $\kappa(w_1,w_2,\lambda'_i)=\kappa(w_1,w_2,\lambda'_0)=0$ or $\kappa(w_3,w_4,\lambda'_i)=\kappa(w_3,w_4,\lambda'_0)=0$. The two possibilities are similar, so we only consider the first one.

By \eqref{eq-kprod}, $\kappa(W_1,W_2,z_0)=0$. On the other hand, using \eqref{eq-kfactor} with $w_k=\eta_k+\eta_k^{-1}$, we can solve the equations to get
\begin{equation}\label{eq-ksplit}
q^{4i+2}\sigma=-\eta_1^{\epsilon_1}\eta_2^{\epsilon_2},\qquad
q^2\sigma=-\eta_1^{\epsilon'_1}\eta_2^{\epsilon'_2},
\end{equation}
where $\epsilon_k,\epsilon'_k=\pm1$. We can use the freedom to choose $\eta_k$ to set $\epsilon'_1=\epsilon'_2=-1$. Now take the quotient of the equations
\begin{equation*}
q^{4i}=\eta_1^{1+\epsilon_1}\eta_2^{1+\epsilon_2}.
\end{equation*}
This cannot be $1$ since $i\ne 0$, so we cannot have $\epsilon_1=\epsilon_2=-1$. Also, if $\epsilon_1=\epsilon_2=1$, then $(\eta_1\eta_2)^{2N}=(q^{4i})^N=1$, which would imply $z_0=\pm2$ via \eqref{eq-ksplit}. Therefore, exactly one of $\epsilon_1,\epsilon_2$ is positive, so we get $q^{4i}=\eta_k^2\ne1$ for one of $k\in\{1,2\}$. Then $\eta_k^{2N}=1$, so $T_N(w_k)=\eta_k^N+\eta_k^{-N}=\pm2$. We also have $\eta_k\ne\pm1$, so $w_k\ne\pm2$. Thus, $z\in\centV_\ram$.

Case 2: $\kappa(w_1,w_2,\lambda'_i)=\kappa(w_3,w_4,\lambda'_0)=0$. By \eqref{eq-kprod}, $\kappa(W_1,W_2,z_0)=\kappa(W_3,W_4,z_0)=0$. This means the classical shadow is a reducible singularity. As before, write $w_k=\eta_k+\eta_k^{-1}$. From $\kappa(w_1,w_2,\lambda'_i)=0$, we can choose $\eta_k$ such that $q^{4i+2}\sigma=-\eta_1\eta_2$. Then similarly, $q^2\sigma=-\eta_3^{-1}\eta_4^{-1}$. Thus, $\eta_1\eta_2\eta_3\eta_4=q^{4i}$ is an $N$-th root of unity not equal to $1$. By construction, $-(\eta_1\eta_2)^N-(\eta_1\eta_2)^{-N}=z_0$, and a straightforward calculation shows that $-(\eta_1\eta_3)^N-(\eta_1\eta_3)^{-N}=z_1$ and $-(\eta_1\eta_4)^N-(\eta_1\eta_4)^{-N}=z_\infty$. Thus, $z\in\centV'_\red$.
\end{proof}

\begin{proof}[Proof of $\Leftarrow$]
We essentially reverse the argument above. For both types of singularities, \eqref{eq-x1inf-04} implies $\Pi_s=\Pi_t=0$. Thus,
\begin{equation*}
\kappa(W_1,W_2,z_0)\kappa(W_3,W_4,z_0)=(z_0^2-4)^2R=0.
\end{equation*}

First suppose $z\in\centV_\ram$. Without loss of generality, we assume $T'_N(w_1)=0$ and $\kappa(W_1,W_2,z_0)=0$. Note $T'_N(w_1)=0$ implies $W_1=\pm2$ and $w_1\ne\pm2$. A quick calculation shows that $K_{1,2}(z_0)=\kappa(W_1,W_2,z_0)$ has a double root at $z_0=\mp W_2$. This also shows $W_2\ne\pm2$, so $w_2\ne\pm2$. Again by $z_0\ne\pm2$, $\sigma^N\ne\pm1$, so $f(\sigma)=K_{1,2}(\epsilon^2(\sigma^N+\sigma^{-N}))$ also has a double root. Note the discriminant of $\kappa_{1,2}(x)=\kappa(w_1,w_2,x)$ is $(w_1^2-4)(w_2^2-4)$, which does not vanish. By the factorization \eqref{eq-kprod}, $\kappa(w_1,w_2,\lambda'_i)$ vanishes for two different $i$. This implies $r_i$ vanish for two different $i$, so the type-$0$ representation is reducible by Lemma~\ref{lemma-red04}.

Now suppose $z\in\centV'_\red$. Then $z=h(\vec{\eta})$ for some $\vec{\eta}\in\cx^4$ such that $\eta_1\eta_2\eta_3\eta_4=\omega$ is an $N$-th root of unity not equal to $1$. We can write $\omega=q^{4k}$ for some $k\ne0$. Since $z_0=-(\eta_1\eta_2)^N-(\eta_1\eta_2)^{-N}$, after possibly relabeling the basis to change the choice of $\sigma$, we can write $\eta_1\eta_2=-q^{-2}\sigma^{-1}$. Then $\eta_3\eta_4=q^{4k}(\eta_1\eta_2)^{-1}=-q^{4k+2}\sigma$. Using \eqref{eq-kfactor}, we get $\kappa(w_1,w_2,\lambda'_0)=\kappa(w_3,w_4,\lambda'_k)=0$. Then $r_0=r_k=0$. Since $k\ne0$, the type-$0$ representation is reducible by Lemma~\ref{lemma-red04}.
\end{proof}

\subsection{Exceptional points}

In this section we construct a 4-dimensional family of representations with $z_0=\pm2$. The representation has a complicated form, so we only use it for exceptional points.

Let $\sigma=\pm1$, and let $\lambda_i$ be given by the same expression as before. If $N$ is odd, let $z_0=\epsilon^2(\sigma^N+\sigma^{-N})$ if $N$ is odd. If $N$ is even, then necessarily $\epsilon^2(\sigma^N+\sigma^{-N})=2$, but we allow both $z_0=\pm2$. For $\vec{w}\in\cx^4$, write $w_i=\eta_i+\eta_i^{-1}$.

\begin{proposition}\label{prop-except04}
Let
\begin{equation}\label{eq-z1inf-except04}
z_1=-(\eta_1\eta_4)^{-N}z_0-\eta_4^{-N}W_2-\eta_1^{-N}W_3,\qquad
z_\infty=-(\eta_1\eta_4)^N-(\eta_1\eta_4)^{-N}.
\end{equation}
Then $z=(z_0,z_1,z_\infty,\vec{w})\in\centV$. Define
\begin{equation}
N_-=\begin{cases}
-N/2,&\text{$N$ is even and $z_0=2$},\\
-(N-1)/2,&\text{otherwise},
\end{cases}\qquad
N_+=N_++N-1.
\end{equation}
Let $V$ be an $N$-dimensional vector space with a basis $\{u_i\mid i=N_-,N_-+1,\dotsc,N_+\}$. Then there exists a representation $\rho:\skein_q(\hsphere)\to\End(V)$ with classical shadow $z$ such that
\begin{align}
\rho(\alpha_0)u_i&=\begin{cases}\lambda_iu_i,&N_-\le i\le0,\\ \lambda_iu_i+u_{-i},&0<i\le N_+.\end{cases}\\
\rho(\alpha_\infty)u_i&=\begin{cases}
s_{N_-}u_{N_-+1}+d_{N_-}u_{N_-}+t_{N_--1}u_{N_+},&i=N_-,\\
s_iu_{i+1}+d_iu_i+t_{i-1}u_{i-1},&N_-<i\le0,\\
\beta_iu_{-i-1}+\bar{d}_iu_{-i}+\beta_{-i}u_{-i+1}+s'_{-i}u_{i-1}+d_iu_i+s'_iu_{i+1},&0<i\le N_+.
\end{cases}
\end{align}
Here $u_{N_--1}=u_{N_++1}=0$,  and
\begin{align}
\mu(x)&=\frac{(x+\eta_1\eta_2)(x+\eta_3\eta_4)(x^{-1}\eta_1+\eta_2)(x^{-1}\eta_4+\eta_3)}{\eta_1\eta_2\eta_3\eta_4}.\\
s_i&=\begin{cases}
v_1\mu(q^{4i+2}\sigma),&i\ne-N/2,\\
v_2\mu'(q^{4i+2}\sigma)+v_3\mu(q^{4i+2}\sigma),&i=N_-=-N/2,
\end{cases}\\
t_{i-1}&=\begin{cases}
v_4\mu(q^{-4i+2}\sigma),&i\ne0,\\
v_5\mu'(q^2\sigma)+v_6\mu(q^2\sigma),&i=0,
\end{cases}\\
s'_i&=\begin{cases}
v_7\mu(q^{4i+2}\sigma),&i\ne-1/2,-1,\\
0,&i=-1/2,-1,
\end{cases}\\
\beta_i&=\begin{cases}
v_8\mu'(q^{4i+2}\sigma)+v_9\mu(q^{4i+2}\sigma),&i\ne-1/2,\\0,&i=-1/2,
\end{cases}
\end{align}
where $v_1,\dotsc,v_9$ are coefficients depending on $i$ and $\sigma$, and $v_1,\dotsc,v_8$ are always nonzero.
\end{proposition}

The proof is given in Appendix~\ref{sec-except-04} by explicit construction. Some coefficients here correspond to coefficients with asterisks in the Appendix. The graph of $\rho(\alpha_\infty)$ (with self-loops removed) is shown in Figures~\ref{fig-GZ04}, where the thick black lines are $t_{-1}$ and $s_{N_-}$, the thick blue lines are $\beta_{\pm i}$, the dashed lines are $\bar{d}_i$, and the remaining regular lines are $s_i$ and $t_{i-1}$.

\begin{figure}
\centering
\begin{subfigure}{\linewidth}
\centering
\begin{tikzpicture}[every loop/.style={}]
\node (N) at (10,1){$0$};
\node (1) at (8,1){$-1$};\node (m1) at (8,-1){$1$};
\node (2) at (6,1){$-2$};\node (m2) at (6,-1){$2$};
\node (d) at (4,1){$\dots$};\node (md) at (4,-1){$\dots$};
\path (1) edge[<->] (2) edge[->] (N);
\path (m1) edge[<->] (m2) edge[thick,blue,->] (N);
\path[bend left,->] (N) edge[thick] (1) edge (m1);
\path[dashed,->] (m1) edge (1);
\path[dashed,->] (m2) edge (2);
\path (2) edge[<->] (d) (m2) edge[<->] (md);
\node (hm1) at (2,1){$N_-+1$};
\node (hp1) at (2,-1){$N_+-1$};
\node (hm) at (0,1){$N_-$};
\node (hp) at (0,-1){$N_+$};
\path[<->] (hm1) edge (d) edge (hm);
\path[<->] (hp1) edge (md) edge (hp);
\path[thick,blue,->] (m1) edge (2);
\path[thick,blue,->] (m2) edge (1);
\draw (3,0)node{$\dots$} (5,0)node{$\dots$};
\path (hp1) edge[thick,blue,->] (hm) edge[dashed,->] (hm1);
\path[->] (hp) edge[thick,blue] (hm1) edge[dashed,bend right] (hm);
\path[->] (hm) edge (hp);
\end{tikzpicture}
\subcaption{$N$ is odd}
\end{subfigure}
\\[1em]
\begin{subfigure}{\linewidth}
\centering
\begin{tikzpicture}[every loop/.style={}]
\node (1) at (8,1){$-\frac{1}{2}$};\node (m1) at (8,-1){$\frac{1}{2}$};
\node (2) at (6,1){$-\frac{3}{2}$};\node (m2) at (6,-1){$\frac{3}{2}$};
\node (d) at (4,1){$\dots$};\node (md) at (4,-1){$\dots$};
\path (1) edge[<->] (2) edge[->] (m1);
\path (m1) edge[<->] (m2);
\path[dashed,->,bend left] (m1) edge (1);
\path[dashed,->] (m2) edge (2);
\path (2) edge[<->] (d) (m2) edge[<->] (md);
\node (hm1) at (2,1){$N_-+1$};
\node (hp1) at (2,-1){$N_+-1$};
\node (hm) at (0,1){$N_-$};
\node (hp) at (0,-1){$N_+$};
\path[<->] (hm1) edge (d) edge (hm);
\path[<->] (hp1) edge (md) edge (hp);
\path[thick,blue,->] (m1) edge (2);
\path[thick,blue,->] (m2) edge (1);
\draw (3,0)node{$\dots$} (5,0)node{$\dots$};
\path (hp1) edge[thick,blue,->] (hm) edge[dashed,->] (hm1);
\path[->] (hp) edge[thick,blue] (hm1) edge[dashed,bend right] (hm);
\path[->] (hm) edge (hp);
\end{tikzpicture}
\caption{$N$ is even and $z_0=-2$}
\end{subfigure}
\\[1em]
\begin{subfigure}{\linewidth}
\centering
\begin{tikzpicture}[every loop/.style={}]
\node (2N) at (-1,1){$N_-$};
\node (1) at (1,1){$N_-+1$};\node (m1) at (1,-1){$N_+$};
\node (2) at (3,1){$N_-+2$};\node (m2) at (3,-1){$N_++1$};
\node (d) at (5,1){$\dots$};\node (md) at (5,-1){$\dots$};
\path (1) edge[<->] (2) edge[->] (2N);
\path (m1) edge[<->] (m2) edge[thick,blue,->] (2N);
\path (2) edge[<->] (d); \path (m2) edge[<->] (md);
\path[->,bend right] (2N) edge[thick] (1) edge (m1);
\path[dashed,->] (m1) edge (1); \path[dashed,->] (m2) edge (2);
\path[thick,blue,->] (m1) edge (2);
\path[thick,blue,->] (m2) edge (1);
\node (hm1) at (7,1){$-2$};\node (hp1) at (7,-1){$2$};
\node (hm) at (9,1){$-1$};\node (hp) at (9,-1){$1$};
\path[<->] (hm1) edge (d) edge (hm);
\path[<->] (hp1) edge (md) edge (hp);
\draw (4,0)node{$\dots$} (6,0)node{$\dots$};
\node (N) at (11,1){$0$};
\path[->] (hp1) edge[thick,blue] (hm) edge[dashed] (hm1);
\path[->] (hp) edge[thick,blue] (hm1) edge[dashed] (hm);
\path[->,bend left] (N) edge[thick] (hm) edge (hp);
\path[<-] (N) edge[thick,blue] (hp) edge (hm);
\end{tikzpicture}
\caption{$N$ is even and $z_0=2$}
\end{subfigure}
\caption{The associated graph of $\rho(\alpha_\infty)$}
\label{fig-GZ04}
\end{figure}
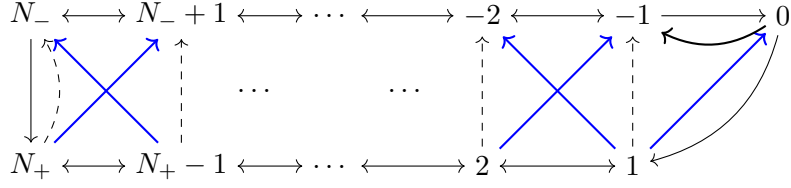
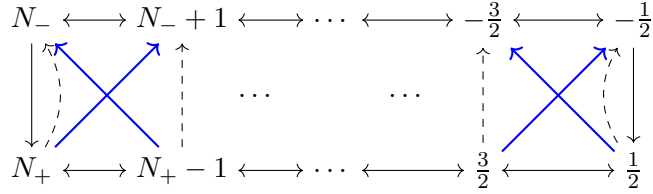
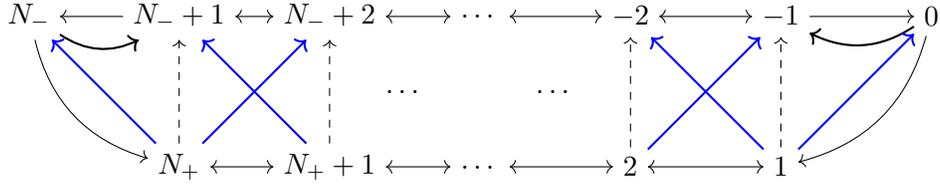

It is possible to show that the representation in the proposition covers all exceptional points up to symmetries of $\centV$, but determining the reducibility for all of them is too much work. We can reduce the cases to check using the fact that the Azumaya locus is open.

\begin{lemma}\label{lemma-except-nonAzu}
If $z\in\azuV^c$ is exceptional, then $z$ is not in the Azumaya locus.
\end{lemma}

\begin{proof}
Let $\exceptV$ denote the set of exceptional points. Combining Lemma~\ref{lemma-type0-red04} with diffeomorphisms of the surface, we know $\azuV^c\setminus\exceptV$ is disjoint from the Azumaya locus, so the closure is too. Since $\exceptV$ is finite and $\azuV^c$ does not have $0$-dimensional components, it is easy to show that every point of $\azuV^c\cap\exceptV$ is in the closure.
\end{proof}

Recall the map $h$ defined at the beginning of this section.

\begin{lemma}\label{lemma-except-remain}
Up to symmetries of $\centV$, the exceptional points not in $\azuV^c$ are $z=(z_0,z_1,z_\infty,\vec{w})$ with
\begin{enumerate}
\item $z_0=z_1=z_\infty=-2$, $w_i=\pm2$ for $i=1,2,3,4$.
\item $z=h(\vec{\eta})$ with $\eta_1\eta_2\eta_3\eta_4=1$ and $\eta_i^{2N}=-1$.
\item $z_0=2$, $z_1=z_\infty=\pm2$, $w_1=2$, $W_2=-2$, $W_3=W_4=-z_1$.
\end{enumerate}
\end{lemma}

\begin{proof}
The exceptional points of $\charV$ are given in Lemma~\ref{lemma-except-04}. Regarding each $W_i$ as the function $T_N(w_i)$, we get all exceptional points of $\centV$.

Consider the first type of exception points in Lemma~\ref{lemma-except-04}, where $z$ is the character of a central representation $r:\pi_1(\hsphere)\to\{\pm I\}$. If one of $w_i\ne\pm2$, then $z\in\centV_\ram\subset\azuV^c$. 
Thus, $z$ is (1) in this case.

Now consider the second type in Lemma~\ref{lemma-except-04}. 
Then $z\in\centV_\red$. By definition, if $z\notin\azuV^c$, then $z\in h(B_1)$. In this case, $z$ is (2).

Finally, consider the last type in Lemma~\ref{lemma-except-04}. Similar to the first type, if $w_1\ne\pm2$, then $z\in\centV_\ram$. If $w_1=\pm2$, then we can use the $H^\ast$ action to set $w_1=2$. We can also permute the punctures to have $W_2=-2$. 
In this case, $z$ is (3).
\end{proof}

\begin{lemma}\label{lemma-red04-2}
If $z\in\centV\setminus\azuV^c$ is exceptional, then $z$ is in the Azumaya locus.
\end{lemma}

\begin{proof}
We just need to check the points given in Lemma~\ref{lemma-except-remain} using the representation $\rho$ in Proposition~\ref{prop-except04}.

For the first point, take $\sigma=-q^{2N}$ and $\eta_i=\pm1$. Then the roots of $\mu(x)$ are $\pm1$. A quick calculation shows that $q^{4i+2}\sigma=\pm1$ if and only if $i=\frac{N-1}{2}$. 
Then the arguments of Lemmas~\ref{lemma-red-2o} and \ref{lemma-red-2e} show that $\rho$ is irreducible.

For the other two types of points, $z_0=2$, so we can take $\sigma=q^{2N}$. Then $(-q^2\sigma)^N=(-\epsilon^2)\epsilon^2=-1$. For the second point, we take $\vec\eta$ such that $z=h(\vec\eta)$. Then $z_0=2$ forces $(\eta_1\eta_2)^N=-1$. For the last point, we have $\eta_1=1$, $\eta_2^N=-1$, and $\eta_3^N=\eta_4^N=\pm1$.

In both cases, $(\eta_1\eta_2)^N=-1$, and exactly one of $(\eta_3\eta_4)^N,(\eta_1\eta_2^{-1})^N,(\eta_4\eta_3^{-1})^N$ is $-1$. This shows $\eta_1\eta_2=-q^{4k+2}\sigma$ for some $k$, and $\mu(q^{4i+2}\sigma)=0$ if and only if $i=k,-k-1$. By exchanging all $\eta_i$ with their inverses if necessarily, we can assume $0\le k\le N_+$. If $k\ne\frac{N-1}{2}$, then we also see that $x=q^{\pm(4k+2)}\sigma$ do not have multiplicity as roots of $\mu(x)$, so $\mu'(x)\ne0$ at these roots.

Now we consider the following cases.

Case 1: $N$ is odd and $k=N_+$. The argument is the same as the first point.

Case 2: $N$ is even and $k=N_+$. Then $-k-1=N_-$, so $t_{N_-}=0$,
$s_{N_-}=v_2\mu'(q^{4N_-+2}\sigma)\ne0$, and all other $s_i,t_i,s'_i$ are nonzero. Moreover, $\beta_{N_+}=v_8\mu'(q^{4N_++2}\sigma)\ne0$. The argument of irreducibility is similar to Lemma~\ref{lemma-red-2e}. Because all $s_i,t_i$ other than $t_{N_-}$ are nonzero, we can conclude $u_{N_-+1},\dotsc,u_0$ are in every invariant subspace. Then starting with $u_0$, we also get $u_1$ since $s_0\ne0$, and we can continue all the way to $u_{N_+}$ since $s'_i\ne0$. Finally, we get $u_{N_-}$ from $u_{N_+}$ using $\beta_{N_+}\ne0$.

Case 3: $k=0$. Then $s_0=0$, $t_{-1}=v_5\mu'(q^2\sigma)\ne0$, and all other $s_i,t_i,s'_i$ are nonzero. Moreover, $\beta_{-1}=v_8\mu'(q^{-2}\sigma)\ne0$. This is similar to Case 2. The details are omitted.

Case 4: $0<k<N_+$. Then $s_k=s_{-k-1}=0$, $t_{-k-1}=t_k=0$, $s'_k=s'_{-k-1}=0$, and all other $s_i,t_i,s'_i$ are nonzero. Moreover, $\beta_k,\beta_{-k-1}\ne0$ by the same calculation as the previous cases. Proceed as Lemma~\ref{lemma-red-2e}. At least one set of $\{u_{N_-},\dotsc,u_{-k-1}\}$ or $\{u_{-k},\dots,u_0\}$ is in any invariant subspace. Take the second possibility as an example. Starting with $u_0$, we can also get $u_1,\dots,u_k$. Since $\beta_k\ne0$, $u_k$ leads to $u_{-k-1}$. Then looping around from $u_{N_-}$, we get all the remaining vectors.
\end{proof}

\appendix

\section{Representations at exception points of the one-punctured torus}

\subsection{Proof of Proposition~\ref{prop-except-odd}}\label{sec-except-odd}

Given $\sigma\in\cx$ not a root of unity and $k$ an integer, let
\begin{equation}
s_i=(q^{2i-k+1}\sigma-q^{-2i+k-1}\sigma^{-1})/\hat\lambda_i,\qquad
t_i=(q^{2i+k+1}\sigma-q^{-2i-k-1}\sigma^{-1})/\hat\lambda_{i+1}.
\end{equation}
Then
\begin{align*}
s_it_i&=\frac{-q^{2k}-q^{-2k}+q^{4i+2}\sigma^2+q^{-4i-2}\sigma^{-2}}{\hat\lambda_i\hat\lambda_{i+1}},\\
\prod_{i=1}^N s_i&=\frac{(q^{-k+1}\sigma)^N-(q^{-k+1}\sigma)^{-N}}{\sigma^N-\sigma^{-N}}=(q^N)^{k-1}=\epsilon^{k-1}
\end{align*}
by Lemma~\ref{lemma-Cheb-prod}, and similarly, $\prod_{i=1}^N t_i=\epsilon^{k-1}$. This defines a type-$0$ representation with $w=-q^{2k}-q^{-2k}$, and \eqref{eq-x1inf-a} shows that the corresponding classical shadow is $(z_0,2\epsilon^kz_0,2\epsilon^{k-1},w)$ with $z_0=\sigma^N+\sigma^{-N}$.

In the limit $\sigma\to1$, the classical shadow becomes $(2,2\epsilon^k,2\epsilon^{k-1},w)$. However, the limits of $s_i,t_i$ do not all exist. We can circumvent this by defining a new basis. Let $\bar{N}=\frac{N-1}{2}$. Fix the range of indices to be $-\bar{N}\le i\le\bar{N}$. Define
\begin{equation}
u_i=v_i,\quad -\bar{N}\le i\le0,\qquad
u_i=\frac{v_i-v_{-i}}{\lambda_i-\lambda_{-i}},\quad 0<i\le\bar{N}.
\end{equation}
Then
\begin{equation}\label{eq-trace2-0-odd}
\rho(\alpha_0)u_i=\begin{cases}
\lambda_iu_i,&-\bar{N}\le i\le0,\\
\lambda_iu_i+u_{-i},&0<i\le\bar{N}.
\end{cases}
\end{equation}
Now consider $\rho(\alpha_\infty)$ with the basis $\{u_i\}$. After some routine calculations, we get
\begin{equation}
\rho(\alpha_\infty)u_i=\begin{cases}
s_{-\bar{N}}u_{-\bar{N}+1}+t_{\bar{N}}u_{-\bar{N}}+t^\ast_{-\bar{N}-1}u_{\bar{N}},&i=-\bar{N},\\
s_iu_{i+1}+t_{i-1}u_{i-1},&-\bar{N}<i<0,\\
s^\ast_0u_1+t^\ast_{-1}u_{-1},&i=0,\\
\beta_iu_{-i-1}+\beta_{-i}u_{-i+1}+t'_{i-1}u_{i-1}+s'_iu_{i+1},&0<i<\bar{N},\\
\beta_{\bar{N}}u_{-\bar{N}}+\beta_{-\bar{N}}u_{-\bar{N}+1}+t'_{\bar{N}-1}u_{\bar{N}-1}-t_{\bar{N}}u_{\bar{N}},&i=\bar{N},
\end{cases}
\end{equation}
where
\begin{gather}
s'_i=\frac{\lambda_{i+1}-\lambda_{-i-1}}{\lambda_i-\lambda_{-i}}s_i,\qquad
t'_{i-1}=\frac{\lambda_{-i+1}-\lambda_{i-1}}{\lambda_{-i}-\lambda_i}t_{i-1},\qquad
\beta_i=\frac{s_i-t_{-i-1}}{\lambda_i-\lambda_{-i}}.
\label{eq-2trace-consts}\\
t^\ast_{-\bar{N}-1}=(\lambda_{\bar{N}}-\lambda_{-\bar{N}})t_{\bar{N}},\qquad
s^\ast_0=(\lambda_1-\lambda_{-1})s_0,\qquad
t^\ast_{-1}=s_0+t_{-1}.
\label{eq-2trace-consts2}
\end{gather}
All coefficients have limits as $\sigma\to1$. Then the limit will define a representation with the expected classical shadow. Specifically, the limits of all $s_i,s'_i,t_{i-1},t'_{i-1}$ appeared can be obtained by substitution. Using $\simeq$ to denote having the same limit as $\sigma\to1$, the remaining coefficients are
\begin{equation*}
\left\{
\begin{aligned}
&t_{i-1}\simeq s_{-i},\quad
t'_{i-1}\simeq s'_{-i},\quad
\beta_i\simeq\frac{2q^{k-1}-2q^{-k+1}}{(q^{2i}-q^{-2i})^3},\quad
\beta_{-i}\simeq-\beta_i,\qquad i\ne0.\\
&t^\ast_{-\bar{N}-1}\simeq0,\quad
t^\ast_{-1}\simeq q^{k-1}+q^{-k+1},\quad
s^\ast_0\simeq(q^2-q^{-2})(q^{-k+1}-q^{k-1}),\quad
-t_{\bar{N}}\simeq s'_{\bar{N}}.
\end{aligned}
\right.
\end{equation*}
This gives the formulas in Proposition~\ref{prop-except-odd}.

\subsection{Proof of Proposition~\ref{prop-except-even}}\label{sec-except-even}

Similar to the odd $n$ case, we can construct type-$0$ representations using
\begin{equation}
s_i=(q^{2i-2k}\sigma-q^{-2i+2k}\sigma^{-1})/\hat\lambda_i,\qquad
t_i=(q^{2i+2k+2}\sigma-q^{-2i-2k-2}\sigma^{-1})/\hat\lambda_{i+1},
\end{equation}
and verify that the classical shadow is $(z_0,z_0,2,-q^{4k+2}-q^{-4k-2})$ with $z_0=\sigma^{2N}+\sigma^{-2N}$. Fix the range of indices to be $-N\le i<N$. Define a new basis
\begin{equation}
u_i=v_i,\quad -N\le i\le0,\qquad
u_i=\frac{v_i-v_{-i}}{\lambda_i-\lambda_{-i}},\quad 0<i<N.
\end{equation}
Then
\begin{equation}
\rho(\alpha_0)u_i=\begin{cases}
\lambda_iu_i,&-N\le i\le 0,\\
\lambda_iu_i+u_{-i},&0<i<N.
\end{cases}
\end{equation}
Now consider $\alpha_\infty$ with the basis $\{u_i\}$. After some routine calculations, we get
\begin{equation}
\rho(\alpha_\infty)u_i=\begin{cases}
s^\ast_{-N}u_{-N+1}+t^\ast_{-N-1}u_{N-1},&i=-N,\\
s_iu_{i+1}+t_{i-1}u_{i-1},&-N<i<0,\\
s^\ast_0u_1+t^\ast_{-1}u_{-1},&i=0,\\
\beta_iu_{-i-1}+\beta_{-i}u_{-i+1}+t'_{i-1}u_{i-1}+s'_iu_{i+1},&0<i<N,
\end{cases}
\end{equation}
where $s'_i,t'_{i-1},\beta_i$ are given by \eqref{eq-2trace-consts}, and
\begin{equation}\label{eq-2trace-consts-ev}
s^\ast_{-N}=s_{-N}+t_{N-1},\qquad
t^\ast_{-N-1}=(\lambda_{N-1}-\lambda_{-N+1})t_{N-1}.
\end{equation}
Note the formula uses a non-existent vector $u_N$, but the coefficient $s'_{N-1}$ is $0$ by definition.

All coefficients have limits as $\sigma\to1$. Then the limit will define a representation as well. The limits of all $s_i,s'_i,t_{i-1},t'_{i-1}$ appeared can be obtained by substitution. Using $\simeq$ to denote having the same limit as $\sigma\to1$, the remaining coefficients are
\begin{equation*}
\left\{
\begin{aligned}
&t_{i-1}\simeq s_{-i},\quad
t'_{i-1}\simeq s'_{-i},\quad
\beta_i\simeq\frac{2q^{2k}-2q^{-2k}}{(q^{2i}-q^{-2i})^3},\quad
\beta_{-i}\simeq-\beta_i\qquad i\ne-N,0.\\
&t^\ast_{-1}=s^\ast_{-N}\simeq q^{2k}+q^{-2k},\qquad
t^\ast_{-N-1}\simeq(q^2-q^{-2})(q^{2k}-q^{-2k})\simeq-s^\ast_0.
\end{aligned}
\right.
\end{equation*}
This agrees with Proposition~\ref{prop-except-even}.

\section{Representations at exception points of the four-punctured sphere}
\label{sec-except-04}

Using \eqref{eq-kfactor}, we can show
\begin{equation}\label{eq-kmu}
\mu(x)\mu(x^{-1})=\kappa(w_1,w_2,x+x^{-1})\kappa(w_3,w_4,x+x^{-1}).
\end{equation}

If $\sigma$ is not a root of unity, let
\begin{equation}
s_i=-\frac{\mu(q^{4i+2}\sigma)}{\hat\lambda'_i\hat\lambda_i},\qquad
t_i=-\frac{\mu(q^{-4i-2}\sigma^{-1})}{\hat\lambda'_i\hat\lambda_{i+1}}.
\end{equation}
By \eqref{eq-kmu}, this defines a type-$0$ representation. Then
\begin{equation*}
\begin{aligned}
\prod_{i=1}^Ns_i&=(-1)^N\frac{(\epsilon^2\sigma^N+\eta_1^N\eta_2^N)(\epsilon^2\sigma^N+\eta_3^N\eta_4^N)(\epsilon^2\sigma^{-N}\eta_1^N+\eta_2^N)(\epsilon^2\sigma^{-N}\eta_4^N+\eta_3^N)}{q^{2N}(z_0^2-4)(\eta_1\eta_2\eta_3\eta_4)^N}\\
&=-\epsilon^2\frac{M(\epsilon^2\sigma^N)}{z_0^2-4},
\end{aligned}
\end{equation*}
where $M(x)$ has the same form as $\mu(x)$ except each $\eta_i$ is replaced by $\eta_i^N$. 
Similarly,
\begin{equation*}
\prod_{i=1}^Nt_i=-\epsilon^2\frac{M(\epsilon^2\sigma^{-N})}{z_0^2-4}.
\end{equation*}
Now we can use Lemma~\ref{lemma-x1inf-04} to verify \eqref{eq-z1inf-except04}.

\subsection{$N$ is odd}\label{sec-except-odd04}

To obtain $z_0=\pm2$, we use the change of basis in Section~\ref{sec-except-odd} and take the limit $\sigma\to\pm1$. Then $\rho(\alpha_0)$ is given by \eqref{eq-trace2-0-odd}. For $\rho(\alpha_\infty)$, we use the notations from \eqref{eq-2trace-consts}, \eqref{eq-2trace-consts2}, and
\begin{equation}\label{eq-2trace-diag}
d^\ast_{\pm\bar{N}}=d_{\pm\bar{N}}\mp t_{\bar{N}},\qquad
\bar{d}_i=\frac{d_i-d_{-i}}{\lambda_i-\lambda_{-i}},\qquad
\bar{d}^\ast_{\bar{N}}=\frac{d_{\bar{N}}-d_{-\bar{N}}+s_{\bar{N}}-t_{\bar{N}}}{\lambda_{\bar{N}}-\lambda_{-\bar{N}}}.
\end{equation}
Then in the new basis,
\begin{equation}
\rho(\alpha_\infty)u_i=\begin{cases}
s_{-\bar{N}}u_{-\bar{N}+1}+d^\ast_{-\bar{N}}u_{-\bar{N}}+t^\ast_{-\bar{N}-1}u_{\bar{N}},&i=-\bar{N},\\
s_iu_{i+1}+d_iu_i+t_{i-1}u_{i-1},&-\bar{N}<i<0,\\
s^\ast_0u_1+d_0u_0+t^\ast_{-1}u_{-1},&i=0,\\
\beta_iu_{-i-1}+\bar{d}_iu_{-i}+\beta_{-i}u_{-i+1}+t'_{i-1}u_{i-1}+d_iu_i+s'_iu_{i+1},&0<i<\bar{N},\\
\bar{d}^\ast_{\bar{N}}u_{-\bar{N}}+\beta_{-\bar{N}}u_{-\bar{N}+1}+t'_{\bar{N}-1}u_{\bar{N}-1}+d^\ast_{\bar{N}}u_{\bar{N}},&i=\bar{N}.
\end{cases}
\end{equation}
We can verify that every coefficient has a limit as $\sigma\to\pm1$. In particular, all $s_i,s'_i,t_{i-1},t'_{i-1},d_i$ appeared are obtained by direct substitution $\sigma=\pm1$. Recall $\simeq$ denotes having the same limit as $\sigma\to\pm1$. Then
\begin{equation}\label{eq-limit-odd}
\left\{
\begin{aligned}
&t_{i-1}\simeq s_{-i},\qquad t'_{i-1}\simeq s'_{-i},\qquad i\ne0,\bar{N},\\
&t^\ast_{-\bar{N}-1}\simeq q^{2N}\sigma\mu(q^{2N}\sigma),\qquad
s^\ast_0\simeq-(q^2+q^{-2})\sigma\mu(q^2\sigma),\\
&t^\ast_{-1}\simeq\frac{q^2+q^{-2}}{(q^2-q^{-2})^2}\mu(q^2\sigma)-\frac{q^2\sigma}{q^2-q^{-2}}\mu'(q^2\sigma),\\
&\beta_i\simeq\frac{2(q^{8i+2}-q^{-8i-2})}{\hat\lambda_i^3(\hat\lambda'_i)^2}\mu(q^{4i+2}\sigma)-\frac{q^{4i+2}\sigma}{\hat\lambda_i^2\hat\lambda'_i}\mu'(q^{4i+2}\sigma).
\end{aligned}
\right.
\end{equation}
The last two are obtained using L'H\^opital's rule, which should be evident from the derivatives. The remaining constants $d^\ast_{\pm\bar{N}},\bar{d}_i,\bar{d}^\ast_{\bar{N}}$ are omitted.

\subsection{$N$ is even and $z_0\to-2$}\label{sec-except-even-m2}

One way to obtain $z_0=-2$ is to take the limit $\sigma\to\pm q^{-2}$. Let $\tilde{\sigma}=q^2\sigma$ and $\tildei=i-1/2$. In all coefficients, $\sigma$ is multiplied by $q^{4i}$. Since $q^{4\tildei}\tilde{\sigma}=q^{4i}\sigma$, we can use half-integer indices in expressions such as $\lambda_{\tildei}$ or $s_{\tildei}$ if we replace $\sigma$ with $\tilde\sigma$. This is the convention used in this section.

Let $\bar{N}=\frac{N-1}{2}$. We fix the range of indices to be $-\bar{N}\le\tildei\le\bar{N}$. Then $\lambda_{\tildei}\simeq\lambda_{-\tildei}$. Define a new basis
\begin{equation}
u_{\tildei}=v_i,\quad -\bar{N}\le\tildei<0,\qquad
u_{\tildei}=\frac{v_i-v_{-i+1}}{\lambda_{\tildei}-\lambda_{-\tildei}},\quad 0<\tildei\le\bar{N}.
\end{equation}
Then
\begin{equation}
\rho(\alpha_0)u_i=\begin{cases}
\lambda_{\tildei}u_{\tildei},&-\bar{N}\le\tildei<0,\\
\lambda_{\tildei}u_{\tildei}+u_{-\tildei},&0<\tildei\le\bar{N}.
\end{cases}
\end{equation}
We use the constants from \eqref{eq-2trace-consts} and \eqref{eq-2trace-diag}, $t^\ast_{-\bar{N}-1}$ from \eqref{eq-2trace-consts2}, and the following constants.
\begin{gather}
s^\ast_{-1/2}=(\lambda_{1/2}-\lambda_{-1/2})s_{-1/2},\qquad
d^\ast_{\pm1/2}=d_{\pm1/2}\mp s_{-1/2},\\
\bar{d}^\ast_{1/2}=\frac{d_{1/2}-d_{-1/2}-s_{-1/2}+t_{-1/2}}{\lambda_{1/2}-\lambda_{-1/2}}.
\end{gather}
Then
\begin{equation}
\rho(\alpha_\infty)u_i=\begin{cases}
s_{-\bar{N}}u_{-\bar{N}+1}+d^\ast_{-\bar{N}}u_{-\bar{N}}+t^\ast_{-\bar{N}-1}u_{\bar{N}},&\tildei=-\bar{N},\\
s_{\tildei}u_{\tildei+1}+d_{\tildei}u_{\tildei}+t_{\tildei-1}u_{\tildei-1},&-\bar{N}<\tildei<-1/2,\\
s^\ast_{-1/2}u_{1/2}+d^\ast_{-1/2}u_{-1/2}+t_{-3/2}u_{-3/2},&\tildei=-1/2,\\
\beta_{1/2}u_{-3/2}+\bar{d}^\ast_{1/2}u_{-1/2}+d^\ast_{1/2}u_{1/2}+s'_{1/2}u_{3/2},&\tildei=1/2,\\
\beta_{\tildei}u_{-\tildei-1}+\bar{d}_{\tildei}u_{-\tildei}+\beta_{-\tildei}u_{-\tildei+1}+t'_{\tildei-1}u_{\tildei-1}+d_{\tildei}u_{\tildei}+s'_{\tildei}u_{\tildei+1},&1/2<\tildei<\bar{N},\\
\bar{d}^\ast_{\bar{N}}u_{-\bar{N}}+\beta_{-\bar{N}}u_{-\bar{N}+1}+t'_{\bar{N}-1}u_{\bar{N}-1}+d^\ast_{\bar{N}}u_{\bar{N}},&\tildei=\bar{N}.
\end{cases}
\end{equation}
The limits in \eqref{eq-limit-odd} still hold (with $\sigma$ replaced by $\tilde\sigma\to\pm1$), and
\begin{equation}
s^\ast_{-1/2}\simeq\tilde\sigma\mu(\tilde\sigma).
\end{equation}

\subsection{$N$ is even and $z_0\to2$}

We use the change of basis in Section~\ref{sec-except-even} with $D=2N$ replaced by $N$. Then in the new basis
\begin{equation}
\rho(\alpha_\infty)u_i=\begin{cases}
s^\ast_{-N/2}u_{-N/2+1}+d_{-N/2}u_{-N/2}+t^\ast_{-N/2-1}u_{N/2-1},&i=-N/2,\\
s_iu_{i+1}+d_iu_i+t_{i-1}u_{i-1},&-N/2<i<0,\\
s^\ast_0u_1+d_0u_0+t^\ast_{-1}u_{-1},&i=0,\\
\beta_iu_{-i-1}+\bar{d_i}u_{-i}+\beta_{-i}u_{-i+1}+t'_{i-1}u_{i-1}+d_iu_i+s'_iu_{i+1},&0<i<N/2,
\end{cases}
\end{equation}
where we used the constants from \eqref{eq-2trace-consts}, \eqref{eq-2trace-consts2}, \eqref{eq-2trace-consts-ev} ($N$ replaced with $N/2$), and \eqref{eq-2trace-diag}. Same as Section~\ref{sec-except-even}, the formula uses a non-existent vector $u_{N/2}$, but the coefficient $s'_{N/2-1}$ is $0$ by definition. The limits in \eqref{eq-limit-odd} still hold, and
\begin{equation*}
\left\{
\begin{aligned}
&s^\ast_{-N/2}\simeq\frac{q^2+q^{-2}}{(q^2-q^{-2})^2}\mu(q^{-2N+2}\sigma)+\frac{q^2\sigma}{q^2-q^{-2}}\mu'(q^{-2N+2}\sigma),\\
&t^\ast_{-N/2-1}\simeq(q^2+q^2)\sigma\mu(q^{-2N+2}\sigma).
\end{aligned}
\right.
\end{equation*}

\printbibliography

@article{BW,
  title={Representations of the Kauffman bracket skein algebra I: invariants and miraculous cancellations},
  author={Bonahon, Francis and Wong, Helen},
  journal={Inventiones mathematicae},
  volume={204},
  pages={195--243},
  year={2016},
  publisher={Springer}
}

@article{BWqtr,
  title={Quantum traces for representations of surface groups in $SL_2(\mathbb{C})$},
  author={Bonahon, Francis and Wong, Helen},
  journal={Geometry \& Topology},
  volume={15},
  number={3},
  pages={1569--1615},
  year={2011},
  publisher={Mathematical Sciences Publishers}
}

@article{FKL1,
  title={Unicity for representations of the Kauffman bracket skein algebra},
  author={Frohman, Charles and Kania-Bartoszynska, Joanna and L{\^e}, Thang},
  journal={Inventiones mathematicae},
  volume={215},
  pages={609--650},
  year={2019},
  publisher={Springer}
}

@article{FKL3,
  title={Sliced skein algebras and geometric Kauffman bracket},
  author={Frohman, Charles and Kania-Bartoszynska, Joanna and L{\^e}, Thang},
  journal={arXiv preprint arXiv:2310.06189},
  year={2023}
}

@book{BG,
  title={Lectures on algebraic quantum groups},
  author={Brown, Ken and Goodearl, Ken R},
  year={2012},
  publisher={Birkh{\"a}user}
}

@article{BY,
  title={Azumaya loci and discriminant ideals of PI algebras},
  author={Brown, Ken A and Yakimov, Milen T},
  journal={Advances in Mathematics},
  volume={340},
  pages={1219--1255},
  year={2018},
  publisher={Elsevier}
}

@inbook{DP,
  title={Quantum groups},
  booktitle={D-modules, representation theory, and quantum groups},
  author={De Concini, C and Procesi, C},
  pages={31--140},
  year={1992}
}

@inproceedings{Ba,
  title={Skein spaces and spin structures},
  author={Barrett, John W},
  booktitle={Mathematical Proceedings of the Cambridge Philosophical Society},
  volume={126},
  number={2},
  pages={267--275},
  year={1999},
  organization={Cambridge University Press}
}

@article{Tak,
  title={Representations of the Kauffman skein algebra of small surfaces},
  author={Takenov, Nurdin},
  journal={arXiv preprint arXiv:1504.04573},
  year={2015}
}

@article{HP,
  title={On the classification of irreducible finite-dimensional representations of $U'_q(so_3)$ algebra},
  author={Havl{\'{i}}{\v{c}}ek, M and Po{\v{s}}ta, Severin},
  journal={Journal of Mathematical Physics},
  volume={42},
  number={1},
  pages={472--500},
  year={2001},
  publisher={American Institute of Physics}
}

@article{Bul,
  title={Rings of $SL_2(\mathbb{C})$-characters and the Kauffman bracket skein module},
  author={Bullock, Doug},
  journal={Commentarii Mathematici Helvetici},
  volume={72},
  number={4},
  pages={521--542},
  year={1997},
  publisher={Springer}
}

@article{Bulfg,
  title={A finite set of generators for the Kauffman bracket skein algebra},
  author={Bullock, Doug},
  journal={Mathematische Zeitschrift},
  volume={231},
  number={1},
  pages={91--101},
  year={1999},
  publisher={Springer}
}

@article{PS,
  title={On Skein Algebras And $SL_2(\mathbb{C})$-Character Varieties},
  author={Przytycki, J{\'o}zef H and Sikora, Adam S},
  journal = {Topology},
  volume = {39},
  number = {1},
  pages = {115-148},
  year = {2000}
}

@article{PS2,
  title={Skein algebras of surfaces},
  author={Przytycki, J{\'o}zef and Sikora, Adam},
  journal={Transactions of the American Mathematical Society},
  volume={371},
  number={2},
  pages={1309--1332},
  year={2019}
}

@article{Pr,
  title={Fundamentals of Kauffman bracket skein modules},
  author={Przytycki, J{\'o}zef H},
  journal = {Kobe Journal of Mathematics},
  volume = {16},
  number = {1},
  pages = {45-66},
  year = {1999}
}

@article{BP,
  title={Multiplicative structure of Kauffman bracket skein module quantizations},
  author={Bullock, Doug and Przytycki, Jozef},
  journal={Proceedings of the American Mathematical Society},
  volume={128},
  number={3},
  pages={923--931},
  year={2000}
}

@article{BG0,
  title={The topology of the relative character varieties of a quadruply-punctured sphere},
  author={Benedetto, Robert L and Goldman, William M},
  journal={Experimental Mathematics},
  volume={8},
  number={1},
  pages={85--103},
  year={1999},
  publisher={Taylor \& Francis}
}

@article{CL,
     author = {Cantat, Serge and Loray, Frank},
     title = {Dynamics on Character Varieties and Malgrange irreducibility of Painlev\'e VI equation},
     journal = {Annales de l'Institut Fourier},
     pages = {2927--2978},
     publisher = {Association des Annales de l{\textquoteright}institut Fourier},
     volume = {59},
     number = {7},
     year = {2009}
}

@inproceedings{Tu,
  title={Skein quantization of Poisson algebras of loops on surfaces},
  author={Turaev, Vladimir G},
  booktitle={Annales scientifiques de l'Ecole normale sup{\'e}rieure},
  volume={24},
  number={6},
  pages={635--704},
  year={1991}
}

@article{BFK,
  title={Understanding the Kauffman bracket skein module},
  author={Bullock, Doug and Frohman, Charles and Kania-Bartoszy{\'n}ska, Joanna},
  journal={Journal of knot theory and its ramifications},
  volume={8},
  number={03},
  pages={265--277},
  year={1999},
  publisher={World Scientific}
}

@article{Mu,
  title={Skein and cluster algebras of marked surfaces},
  author={Muller, Greg},
  journal={Quantum topology},
  volume={7},
  number={3},
  pages={435--503},
  year={2016}
}

\end{document}